\documentclass[a4paper,reqno]{amsart}

\usepackage{amsmath}
\usepackage{amsfonts}
\usepackage{amssymb}
\usepackage{amsthm}
\usepackage{mathrsfs}
\usepackage[colorlinks]{hyperref}
\usepackage{color}
\usepackage[all]{xy}
\usepackage{enumerate}
\usepackage{bm}
\usepackage{bbm}
\usepackage{parskip}

\allowdisplaybreaks
\numberwithin{equation}{section}

\newtheorem{theorem}{Theorem}[section]
\newtheorem{proposition}[theorem]{Proposition}
\newtheorem{lemma}[theorem]{Lemma}
\newtheorem{corollary}[theorem]{Corollary}

\theoremstyle{definition}
\newtheorem{definition}[theorem]{Definition}
\newtheorem{definitions}[theorem]{Definitions}

\theoremstyle{remark}
\newtheorem{remark}[theorem]{Remark}
\newtheorem{example}[theorem]{Example}
\newtheorem{examples}[theorem]{Examples}

\makeatletter
\def\thm@space@setup{%
	\thm@preskip=\parsep \thm@postskip=0pt
}
\makeatother

\newcommand\G{\mathcal{G}}

\newcommand\Z{\mathcal{Z}}

\DeclareMathOperator{\Span}{span}
\DeclareMathOperator{\supp}{supp}
\DeclareMathOperator{\Iso}{Iso}

\DeclareMathOperator{\Char}{char}

\DeclareMathOperator{\card}{card}

\DeclareMathOperator{\domm}{dom}
\DeclareMathOperator{\codd}{cod}

\DeclareMathOperator{\Cl}{Cl}

\newcommand{\g}{\gamma}
\newcommand{\id}{\operatorname{id}}

\newcommand{\ep}{\varepsilon}

\newcommand{\ZZ}{\mathbb{Z}}

\newcommand{\im}{\operatorname{im}}

\newcommand{\dom}{\bm{d}}
\newcommand{\cod}{\bm{c}}
\newcommand{\mult}{\bm{m}}
\newcommand{\inv}{\bm{i}}

\usepackage[english]{babel}
\title[The groupoid approach to Leavitt path algebras]{The groupoid approach to Leavitt path algebras}

\author[S. W. Rigby]{Simon W. Rigby}

\address{Department of Mathematics and
Applied Mathematics, University of Cape Town, South~Africa.\newline \newline
Department of Mathematics: Algebra and Geometry, Ghent University, Belgium.}

\email{simon.rigby@ugent.be}

\begin{document}

\begin{abstract}
	When the theory of Leavitt path algebras was already quite advanced, it was discovered that some of the more difficult questions were susceptible to a new approach using topological groupoids. The main result that makes this possible is that the Leavitt path algebra of a graph is graded isomorphic to the Steinberg algebra of the graph's boundary path groupoid. 
	
	This expository paper has three parts: Part 1 is on the Steinberg algebra of a groupoid, Part 2 is on the path space and boundary path groupoid of a graph, and Part 3 is on the Leavitt path algebra of a graph. It is a self-contained reference on these topics, intended to be useful to beginners and experts alike. While revisiting the fundamentals, we prove some results in greater generality than can be found elsewhere, including the uniqueness theorems for Leavitt path algebras.
\end{abstract}

\maketitle

Leavitt path algebras are $\ZZ$-graded algebras with involution, whose generators and relations are encoded in a directed graph. Steinberg algebras, on the other hand, are algebras of functions defined on a special kind of topological groupoid, called an ample groupoid. To understand how they are related, it is useful to weave together some historical threads. This historical overview might not be comprehensive, but it is intended to give some idea of the origins of our subject.
%
%
%
%
\subsection{Historical overview: Groupoids, graphs, and their algebras}

In the late 1950s and early 1960s, William G. Leavitt \cite{leavitt1962module,leavitt1965module} showed that there exist simple rings whose finite-rank free modules admit bases of different sizes. In a seemingly unrelated development, in 1977, Joachim Cuntz \cite{cuntz1977simplec} showed that there exist separable $C^*$-algebras that are simple and purely infinite. Cuntz's paper was one of the most influential in the history of operator theory. It provoked intense interest (that is still ongoing) in generalising, classifying, and probing the structure of various classes of $C^*$-algebras. One of the next landmarks was reached in 1980, when Jean Renault \cite{renault1980groupoid} defined groupoid $C^*$-algebras, taking inspiration from the $C^*$-algebras that had previously been associated to transformation groups. The Cuntz algebras were interpreted as groupoid $C^*$-algebras, and from that point onwards there was a new framework and some powerful results with which to pursue new and interesting examples.

In 1997, Kumjian, Pask, Raeburn, and Renault \cite{kumjian1997graphs} showed how to construct a Hausdorff ample groupoid (and hence a $C^*$-algebra) from a row- and column-finite directed graph with no sinks.  They showed that these $C^*$-algebras universally satisfy the Cuntz-Krieger relations from \cite{Cuntz1980}, which had become significant in the intervening years. Graph $C^*$-algebras were then studied in depth. Usually, they were conceptualised in terms of the partial isometries that generate them; direct methods, rather than groupoid methods, were used predominantly~\cite{bates2000c, raeburn2005graph}.
Meanwhile, the Cuntz algebras had also been interpreted as inverse semigroup $C^*$-algebras. Paterson~\cite{paterson1999groupoids,paterson2002graph}, at the turn of the 21st century, organised the situation a bit better. He showed that all graph $C^*$-algebras (of countable graphs, possibly with sinks, infinite emitters, and infinite receivers) are inverse semigroup $C^*$-algebras, and that all inverse semigroup $C^*$-algebras are groupoid $C^*$-algebras. The key innovation was defining the universal groupoid of an inverse semigroup, which is an ample but not necessarily Hausdorff topological groupoid.

It is unclear when the dots were first connected between Leavitt's algebras and Cuntz's $C^*$-algebras (probably in \cite{ara2002k}, a very long time after they first appeared). The Cuntz algebra $\mathcal{O}_n$ is the norm completion of the complex Leavitt algebra $L_{n,\mathbb{C}}$. Over any field $\mathbb{K}$, the ring $L_{n,\mathbb{K}}$ and the $C^*$-algebra $\mathcal{O}_n$ are purely infinite simple, and they have the same $K_0$ group (but these concepts have a different meaning for rings compared to $C^*$-algebras). This begins a process in which the algebraic community generalises, classifies, and probes the structure of various classes of rings in much the same way as the operator algebra community did with $C^*$-algebras.
Leavitt path algebras were introduced in \cite{abrams2005leavitt} and \cite{ara2007nonstable} as universal $\mathbb{K}$-algebras satisfying path algebra relations and Cuntz-Krieger relations. Generalising the relationship between the Leavitt and Cuntz algebras, the graph $C^*$-algebra of a graph $E$ is the norm completion of the complex Leavitt path algebra of $E$. The interplay with $C^*$-algebras is not the only connection between Leavitt path algebras and other, older, areas of mathematics -- see for instance \cite[\S1]{decade} and \cite{pardo}.

Knowing what we know now, the next step was very natural. Is there a way of defining a ``groupoid $\mathbb{K}$-algebra" in such a way that:
\begin{itemize}
	\item When the input is the universal groupoid of an inverse semigroup $S$, the output is the (discrete) inverse semigroup algebra $\mathbb{K}S$;
	\item When the input is a graph groupoid $\G_E$, the output is the Leavitt path algebra $L_{\mathbb{K}}(E)$?
\end{itemize}

This question was asked and answered by Steinberg \cite{steinberg2010groupoid}, and Clark, Farthing, Sims, and Tomforde~\cite{clark2014groupoid}. Consistent with previous experiences of converting operator algebra constructions into $\mathbb{K}$-algebra constructions, they found that the groupoid $C^*$-algebra is the norm completion of the groupoid $\mathbb{C}$-algebra.  It is also worth noting that Steinberg chose a broad scope and defined groupoid $R$-algebras over any commutative ring $R$, rather than just fields. We call these groupoid algebras \textit{Steinberg algebras}.

Each of the three parts in this paper can be read separately. However, we work towards Leavitt path algebras as the eventual subject of interest, and this influences the rest of the text. For example, in Part \ref{CHAPTER: STEINBERG ALGEBRAS} we try not to impose the Hausdorff assumption on groupoids if it is not necessary, but there are no examples here of non-Hausdorff groupoids. Throughout, we use the graph theory notation and terminology that is conventional in Leavitt path algebras. (In most of the $C^*$-algebra literature, the orientation of paths is reversed.) And in Part \ref{CHAPTER: GRAPHS AND GROUPOIDS}, we ignore some topics like amenability that would be important if we were intending to study the $C^*$-algebras of boundary path groupoids.

A standing assumption throughout the paper is that $R$ is a commutative ring with 1. We rarely need to draw attention to it or require it to be anything special.

\subsection{Background: Leavitt path algebras}
For an arbitrary graph $E$, there is an $R$-algebra, $L_R(E)$, called the Leavitt path algebra of $E$. The role of the graph may seem unclear at the outset, because all it does is serve as a kind of notational device for the generators and relations that define $L_R(E)$. Surprisingly, it turns out that many of the ring-theoretic properties of $L_R(E)$ are controlled by graphical properties of $E$. For example, the Leavitt path algebra has some special properties if the graph is acyclic, cofinal, downward-directed, has no cycles without exits, etc.

%

Since 2005, there has been an abundance of research on Leavitt path algebras. One of the main goals has been to characterise their internal properties, ideals, substructures, and modules. As a result, we have a rich supply of algebras with ``interesting and extreme properties" \cite{alahmedi2013structure}. This is useful for generating counterexamples to reasonable-sounding conjectures, e.g.\! \cite{abrams2014prime,hazrat2018baer}, or for supporting other long-standing conjectures by showing they hold within this varied class, e.g.\! \cite{ambily2018simple,pino2012regular}.

Another goal has been finding invariants that determine Leavitt path algebras up to isomorphism, or Morita equivalence. This enterprise is known as the classification question for Leavitt path algebras. Of course, something that is easier than classifying \textit{all} Leavitt path algebras is classifying those that have a certain property (like purely infinite simplicity), or classifying the Leavitt path algebras of small graphs. This has led to interesting developments in  $K$-theory (see \cite[\S6.3]{LPAbook} and \cite{hazrat2013graded})
and has motivated the study of substructures of Leavitt path algebras, like the socle \cite{pino2010socle} and invariant ideals \cite{kanuni2017classification}.

A third goal is to explain why graph $C^*$-algebras and Leavitt path algebras have so much in common. One expects \textit{a priori} that these two different structures would have little to do with one another. But in fact, many theorems about Leavitt path algebras resemble theorems about graph $C^*$-algebras  \cite[Appendix~1]{decade}. For instance, the graphs whose $C^*$-algebras are $C^*$-simple are exactly the same graphs whose Leavitt path algebras are simple (over any base field).   One conjecture in this general direction is the Isomorphism Conjecture for Graph Algebras \cite{abrams2011isomorphism}: if $E$ and $F$ are two graphs such that $L_\mathbb{C}(E) \cong L_\mathbb{C}(F)$ as rings, then $C^*(E) \cong C^*(F)$ as $*$-algebras. In the unital case, an affirmative answer has been given in \cite[Theorem~14.7]{eilers2016complete}.

\subsection{Background: Steinberg algebras}

An ample groupoid is a special kind of locally compact topological groupoid. The Steinberg algebra of such a groupoid is an $R$-module of functions defined on it. It becomes an associative $R$-algebra once it is equipped with a generally noncommutative operation called the convolution (generalising the multiplicative operation on a group algebra). If the groupoid $\G$ is Hausdorff, one can characterise its Steinberg algebra quite succinctly as the convolution algebra of locally constant, compactly supported functions  $f:~\G \to R$. Steinberg algebras first appeared independently in \cite{steinberg2010groupoid} and~\cite{clark2014groupoid}.
The primary motivation was to generalise other classes of algebras, especially inverse semigroup algebras and Leavitt path algebras.

Steinberg algebras do not only unify and generalise some seemingly disparate classes of algebras, but they also provide an entirely new approach to studying them. Many theorems about Leavitt path algebras and inverse semigroup algebras have since been recovered as specialisations of more general theorems about Steinberg algebras. For example, various papers \cite{orloff2016using,steinberg2016simplicity,steinberg2018chain,steinberg2018prime} have used groupoid techniques to characterise, in terms of the underlying graph or inverse semigroup, when a Leavitt path algebra or inverse semigroup algebra is (semi)prime, indecomposable, (semi)primitive, noetherian, or artinian.

Simplicity theorems play a very important role in graph algebras and some related classes of algebras. (In contrast, inverse semigroup algebras are never simple.) This theme goes right back to the beginning, when Leavitt proved in \cite{leavitt1965module} that the Leavitt algebras $L_{n,\mathbb{K}}$ ($n \ge 2$) are all simple. Likewise, Cuntz proved in \cite{cuntz1977simplec} that the Cuntz algebras $\mathcal{O}_n$ ($n \ge 2$) are $C^*$-simple in the sense that they have no closed two-sided ideals. When Leavitt path algebras were introduced, in the very first paper on the subject, Abrams and Aranda Pino \cite{abrams2005leavitt} wrote the simplicity theorem for Leavitt path algebras of row-finite graphs. It was extended to Leavitt path algebras of arbitrary graphs, as soon as these were defined in \cite{abrams2008leavitt}. 

Once Steinberg algebras appeared on the scene, Brown, Clark, Farthing, and Sims \cite{brown2014simplicity} proved a simplicity theorem for Steinberg algebras of Hausdorff ample groupoids over $\mathbb{C}$. That effort led them to unlock a remarkable piece of research in which they derived a simplicity theorem for the $C^*$-algebras of second-countable, locally compact, Hausdorff \'etale groupoids. It speaks to the significance of these new ideas, that they were put to use in solving a problem that was open for many decades. The effort has recently been repeated for non-Hausdorff groupoids, in \cite{clark2018simplicity}, where it is said that ``We view Steinberg algebras as a laboratory for finding conditions to characterize $C^*$-simplicity for groupoid $C^*$-algebras."

Besides the ones we have already discussed, there are many  interesting classes of algebras that appear as special cases of Steinberg algebras. These include partial skew group rings associated to topological partial dynamical systems \cite{beuter2017interplay}, and Kumjian-Pask algebras associated to higher-rank graphs \cite{clark2017kumjian}. In quite a different application, Nekrashevych  \cite{nekrashevych2016growth} has produced Steinberg algebras with prescribed growth properties, including the first examples of simple algebras of arbitrary Gelfand-Kirrilov dimension.

\subsection{Background: Graph groupoids}
There are actually a few ways to associate a groupoid to a graph $E$; see for example \cite[p. 511]{kumjian1997graphs}, \cite[pp. 156--159]{paterson1999groupoids}, and \cite[Example 5.4]{clark2017cohn}. The one that we are interested in is called the boundary path groupoid, $\G_E$. Its unit space is the set of all paths that are either infinite or end at a sink or an infinite emitter (i.e., boundary paths). This groupoid was introduced in its earliest form, for row- and column-finite graphs without sinks, by Kumjian, Pask, Raeburn, and Renault~\cite{kumjian1997graphs}. It bears a resemblance to a groupoid studied a few years earlier by Deaconu \cite{deaconu1995groupoids}. The construction was later generalised in a number of different directions, taking a route through inverse semigroup theory \cite{paterson2002graph}, and going as far as topological higher-rank graphs (e.g. \cite{kumjian2000higher,renault2018uniqueness, yeend2007groupoid}).

The boundary path groupoid is an intermediate step towards proving that all Leavitt path algebras are Steinberg algebras, and it becomes an important tool for the analysis of Leavitt path algebras. For an arbitrary graph $E$, there is a $\mathbb{Z}$-graded isomorphism $A_R(\G_E)\cong L_R(E)$, where $A_R(\G_E)$ is the Steinberg algebra of $\G_E$ and $L_R(E)$ is the Leavitt path algebra of $E$.  Consequently, if we understand some property of Steinberg algebras (for example, the centre \cite{clark2016using}) then we can understand that property of Leavitt path algebras by translating groupoid terms into graphical terms and applying the isomorphism $A_R(\G_E) \cong L_R(E)$.
Similarly, there is an isometric $*$-isomorphism $C^*(\G_E) \cong C^*(E)$, where $C^*(\G_E)$ is the full groupoid $C^*$-algebra of $\G_E$ and $C^*(E)$ is the graph $C^*$-algebra of $E$.
 
The diagonal subalgebra of a Steinberg algebra (resp., groupoid $C^*$-algebra) is the commutative subalgebra (resp., $C^*$-subalgebra) generated by functions supported on the unit space.
If two ample groupoids $\mathcal{F}$ and $\mathcal{G}$ are topologically isomorphic, it is immediate that $A_R(\mathcal{F}) \cong A_R(\mathcal{G})$ and the isomorphism sends the diagonal to the diagonal. The converse is a very interesting and current research topic called ``groupoid reconstruction". It was shown in \cite{ara2017reconstruction} that if $\mathcal{F}$ and $\mathcal{\G}$ are topologically principal, and $R$ is an integral domain, then $A_R(\mathcal{F}) \cong A_R(\mathcal{G})$ with an isomorphism that preserves diagonals if and only if $\mathcal{\mathcal{F}} \cong \mathcal{\G}$. This was generalised in \cite{carlsen2017diagonal} and \cite{steinberg2017diagonal}. For $C^*$-algebras, there are results of a similar flavour \cite{renault2008cartan} .

For boundary path groupoids, groupoid reconstruction is essentially the question: if $E$ and $F$ are graphs such that $L_R(E) \cong L_R(F)$, does it imply $\G_E \cong \G_F$? 
Many mathematicians \cite{ara2017reconstruction, brown2017diagonal,carlsen2018isomorphism,  steinberg2017diagonal} have been working on this and they have given positive answers after imposing various assumptions on the graphs, the ring $R$, or the type of isomorphism between the Leavitt path algebras. It seems likely that more results will emerge. It is already known from \cite[Theorem 5.1]{brownlowe2017graph} that if there exists a diagonal-preserving isomorphism of graph $C^*$-algebras $C^*(E) \cong C^*(F)$, then $\G_E \cong \G_F$. It is plausible that the groupoid reconstruction programme for graph groupoids could eventually prove the general Isomorphism Conjecture for Graph Algebras~\cite{abrams2011isomorphism}.

\section{The Steinberg algebra of a groupoid} \label{CHAPTER: STEINBERG ALGEBRAS}

Part 1 is structured as follows. It begins, in \S\ref{grp concepts} by providing some background on groupoids. In \S\ref{top grp concepts}, we develop some facts about topological groupoids and almost immediately specialise to \'etale and ample groupoids. We give a very brief treatment of inverse semigroups and their role in the subject. In \S\ref{stein alg}, we introduce the {Steinberg algebra} of an ample groupoid, describing it in a few different ways to make the definition more transparent. We develop the basic theory in a self-contained way, paying attention to what can and cannot be said about non-Hausdorff groupoids. In \S\ref{properties}, we investigate some important properties, showing that these algebras are locally unital and enjoy a kind of symmetry that comes from an involution (in other words, they are $*$-algebras). In \S\ref{first examples}, we investigate the effects of groupoid-combining operations like products, disjoint unions, and directed unions, and find applications with finite-dimensional Steinberg algebras and the Steinberg algebras of approximately finite groupoids.  In \S\ref{gr grp}, we discuss graded groupoids and graded Steinberg algebras.


\subsection{Groupoids} \label{grp concepts}

This classical definition of a groupoid is modified from \cite{renault1980groupoid}. We have chosen to paint a complete picture; indeed, some parts of the definition can be derived from other parts.

\begin{definition} \label{groupoid def}
	\index{$\G$, $\G^{(0)}$, $\G^{(2)}$, $\dom$, $\cod$, $\mult$, $\inv$}
	A \textbf{groupoid} is a system $(\G,\G^{(0)}, \dom, \cod, \mult, \inv)$ such that:
	\begin{enumerate}[(G1)]
		\item $\G$ and $\G^{(0)}$ are nonempty sets, called the \textit{underlying set} and \textit{unit space}, respectively;
		\item $\dom, \cod$ are maps $\G \to \G^{(0)}$, called \textit{domain} and \textit{codomain};
		\item $\mult$ is a partially defined binary operation on $\G$ called \textit{composition}: specifically, it is a map from the set of \textit{composable pairs} 
		\[
		\G^{(2)} = \big\{
		(g,h) \in \G \times \G \mid \dom(g)=\cod(h)
		\big\}
		\] onto $\G$, written as $\mult(g,h) = gh$, with the properties:
		\begin{itemize}
			\item $\dom(gh) = \dom(h)$ and $\cod(gh) = \cod(g)$ whenever the composition $gh$ is defined;
			\item $(gh)k=g(hk)$ whenever either side is defined;
		\end{itemize}
		\item For every $x \in \G^{(0)}$ there is a unique identity $1_x \in \G$ such that $1_xg=g$ whenever $\cod(g) = x$, and $h1_x = h$ whenever $\dom(h)=x$;
		\item $\inv: \G \to \G$ is a map called \textit{inversion}, written as $\inv(g) = g^{-1}$, such that $g^{-1}g = 1_{\cod(g)}$, $gg^{-1} = 1_{\dom(g)}$, and $(g^{-1})^{-1} = g$.
	\end{enumerate}
\end{definition}

The definition can be summarised by saying: {a \textit{groupoid} is a small category in which every morphism is invertible}. Having said this, the elements of $\G$ will usually be called morphisms.

\begin{remark} \label{identify}
	We always identify $x \in \G^{(0)}$ with $1_x \in \G$, so $\G^{(0)}$ is considered a subset of $\G$. The elements of $\G^{(0)}$ are called \textit{units}.
\end{remark}

Many authors write $\bm s$ (source) and $\bm r$ (range) instead of $\dom$ and $\cod$ in the definition of a groupoid. Our notation is chosen to avoid confusion in the context of graphs, where $s$ and $r$ refer to the source and range, respectively, of edges and directed paths.

A \textit{homomorphism} between groupoids $\G$ and $\mathcal{H}$ is a functor $F: \G \to \mathcal{H}$; that is, a map sending units of $\G$ to units of $\mathcal{H}$ and mapping all the morphisms in $\G$ to morphisms in $\mathcal{H}$ in a way that respects the structure.
A \textit{subgroupoid} is a subset $\mathcal{S} \subseteq \mathcal{G}$ that is a groupoid with the structure that it inherits from $\G$. For $x \in \G^{(0)}$, we use the notation ${x\G} = \cod^{-1}(x)$, $\G x = \dom^{-1}(x)$, and ${ x\G y}= \cod^{-1}(x) \cap \dom^{-1}(y)$.  The set $x\G x$ is a group, called the \textit{isotropy group} based at $x$, and the set $\Iso(\G)=\bigcup_{x \in \G^{(0)}} {x\G x}$ is a subgroupoid, called the \textit{isotropy subgroupoid} of $\G$. \index{$\operatorname{Iso}(\G)$} If $\Iso(\G) = \G^{(0)}$ then $\G$ is called \textit{principal}. We say that $\G$ is \textit{transitive} if for every pair of units $x, y \in \G^{(0)}$ there is at least one morphism in $x\G y$.

The \textit{conjugacy class} of $g\in \Iso(\G)$ is the set $
\Cl_\G(g) = \big\{hgh^{-1} \mid h \in {\G{\cod(g)}}  \big\}$. 
The set of conjugacy classes partitions $\Iso(\G)$. The conjugacy class of a unit is called an \textit{orbit}, and the set of orbits partitions $\G^{(0)}$. Equivalently,  the orbit of $x \in \G^{(0)}$ is $\Cl_\G(x)  = \cod(\dom^{-1}(x)) = \dom(\cod^{-1}(x))$, or the unit space of the maximal transitive subgroupoid containing $x$. A subset $U\subseteq \G^{(0)}$ is \textit{invariant} if for all $g \in \G$, $\dom(g) \in U$ implies $\cod(g) \in U$, which is to say that $U$ is a union of orbits. If $x,y \in \G^{(0)}$ belong to the same orbit, then the isotropy groups $x\G x$ and $y\G y$ are isomorphic. In fact, there can be many isomorphisms $x\G x \to {y\G y}$. For every $g \in {y\G x}$ there is an ``inner" isomorphism $x\G x \to {y\G y}$ given by $x \mapsto gxg^{-1}$. This allows us to speak of the isotropy group of an orbit.

\begin{examples} \label{groupoid ex}
	Many familiar mathematical objects are essentially groupoids:
	
	\begin{enumerate}[(a)]
		\item Any \textbf{group} $G$ with identity $\varepsilon$ can be viewed as a groupoid with unit space $\{\varepsilon\}$. Conjugacy classes are conjugacy classes in the usual sense.
		\item If $\{G_i \mid i \in I \}$ is a family of groups with identities $\{\varepsilon_i\mid i \in I\}$, then the disjoint union $\bigsqcup_{i \in I} {G}_i$ has a groupoid structure with $\dom(g) = \cod(g) = \varepsilon_{i}$ for every $g \in G_i$. The composition, defined only for pairs $(g, h)\in \bigsqcup_{i \in I}G_i \times G_i$, is just the relevant group law. This is known as a \textbf{bundle of groups}. The isotropy subgroupoid of any groupoid is a bundle of groups.
		\item \label{pairs} Let $X$ be a set with an equivalence relation $\sim$. We define the \textbf{groupoid of pairs} $\G_X =  \{(x, y) \in X \times X \mid x \sim y\}$ with unit space $X$, and view $(x,y)$ as a morphism with domain $y$, codomain $x$, and inverse $(x,y)^{-1} = (y,x)$.  A pair of morphisms $(x,y), (w, z)$ is composable if and only if $y = w$, and composition is defined as $(x,y)(y,z) = (x,z)$. Every principal groupoid is isomorphic to a groupoid of pairs. If $\sim$ is the indiscrete equivalence relation (where $x \sim y$ for all $x, y \in X$) then $\G_X$ is called the \textit{transitive principal groupoid on $X$}.
		\item \label{transformation groupoid} Let $G$ be a group with a left action on a set $X$. There is a groupoid structure on $G \times X$, where the unit space is $\{\varepsilon\} \times X$, or simply just $X$. We understand that the morphism $(g, x)$ has domain $g^{-1}x$ and codomain $x$. Composition is defined as $(g, x)(h, g^{-1}x) = (g h, x)$, and inversion as $(g, x)^{-1} = (g^{-1}, g^{-1} x)$. The isotropy group at $x$ is isomorphic to the stabiliser subgroup associated to $x$. Orbits are orbits in the usual sense, and the groupoid is transitive if and only if the action is transitive.  This is called the \textbf{transformation groupoid} associated to the action of $G$ on $X$. 
		\item The \textbf{fundamental groupoid} of a topological space $X$ is the set of homotopy path classes on $X$. The unit space of this groupoid is $X$ itself, and the isotropy group at $x \in X$ is the fundamental group $\pi_1(X,x)$. The groupoid is transitive if and only if $X$ is path-connected, and it is principal if and only if every path component is simply connected.
	\end{enumerate}
	
\end{examples}

\subsection{Topological groupoids} \label{top grp concepts}

Briefly, here are some of our topological conventions.
We use the word \textit{base} to mean a collection of open sets, called \textit{basic open sets}, that generates a topology by taking unions. A neighbourhood base is a filter for the set of neighbourhoods of a point. In this paper, the word \textit{basis} is reserved  for linear algebra. A \textit{compact} topological space is one in which every open cover has a finite subcover, and a \textit{locally compact} topological space is one in which every point has a neighbourhood base of compact sets. If $X$ and $Y$ are topological spaces, a \textit{local homeomorphism} is a map $f: X \to Y$ with the property: every point in $X$ has an open neighbourhood $U$ such that $f|_U$ is a homeomorphism onto an open subset of $Y$. Every local homeomorphism is open and continuous.

The definition of a topological groupoid is straightforward, but there is some inconsistency in the literature on what it means for a groupoid to be \'etale or locally compact. While some papers require germane conditions, our definitions are chosen to be classical and minimally restrictive. 
We are mainly concerned with \'etale and ample groupoids. Roughly speaking, \'etale groupoids are topological groupoids whose topology is locally determined by the unit space.

\begin{definition} A groupoid $\G$ is
	\begin{enumerate}[(a)]
		\item a \textbf{topological groupoid} if its underlying set has a topology, and the maps $\mult$ and $\inv$ are continuous, with the understanding that $\G^{(2)}$ inherits its topology from $\G\times \G$;
		\item an \textbf{\'etale groupoid} if it is a topological groupoid and $\dom$ is a local homeomorphism.
	\end{enumerate}
\end{definition}


Some pleasant consequences follow from these two definitions. In any topological groupoid, $\inv$ is a homeomorphism because it is a continuous involution, and $\dom$ and $\cod$ are both continuous because $\dom(g) = \mult(\inv(g),g)$ and $\cod = \dom \inv$. If $\G$ is \'etale, then $\dom$, $\cod$, and $\mult$ are local homeomorphisms, and $\G^{(0)}$ is open in $\G$ (the openness of $\G^{(0)}$ is proved from first principles in \cite[Proposition 3.2]{exel2008inverse}). If $\G$ is a Hausdorff topological groupoid, then $\G^{(0)}$ is closed. Indeed (and this neat proof is from \cite{sims2017etale}) if $(x_i)_{i \in I}$ is a net in $\G^{(0)}$ with $x_i \to g \in \G$, then $x_i = \cod(x_i) \to \cod(g)$ because $\cod$ is continuous, so $g = \cod(g) \in \G^{(0)}$ by uniqueness of limits. If $\G$ is any topological groupoid, the maps $\dom \times \cod : \G \times \G \to \G^{(0)} \times \G^{(0)}$ and $(\dom, \cod): \G \to \G^{(0)} \times \G^{(0)}$ are both continuous. If $\G^{(0)}$ is Hausdorff, the diagonal $\Delta = \{(x,x)\mid x \in \G^{(0)} \}$ is closed in $\G^{(0)} \times \G^{(0)}$; consequently, $\G^{(2)} = (\dom \times \cod) ^{-1}(\Delta)$ is closed in $\G\times \G$ and $\Iso(\G) = (\dom, \cod)^{-1}(\Delta)$ is closed in $\G$.

Let $\G$ be a topological groupoid. If  $U \subseteq \G$ is an open set such that $\cod{|_U}$ and $\dom{|_U}$ are homeomorphisms onto open subsets of $\G^{(0)}$, then $U$ is called an \textit{open bisection}.
If $\G$ is \'etale and $U \subseteq \G$ is open, the restrictions $\cod|_U$ and $\dom|_U$ are continuous open maps, so they need only be injective for $U$ to be an open bisection.
An equivalent definition of an \textit{\'etale groupoid} is a topological groupoid that has a base of open bisections. If $\G$ is \'etale and $\G^{(0)}$ is Hausdorff, then $\G$ is locally Hausdorff, because all the open bisections are homeomorphic to subspaces of $\G^{(0)}$.
Another property of \'etale groupoids is that for any $x \in \G^{(0)}$, the fibres $x\G$ and $\G x$ are discrete spaces. Consequently, a groupoid with only one unit (i.e., a group) is \'etale if and only if it has the discrete topology.

\begin{definition}
	An \textbf{ample groupoid} is a topological groupoid with Hausdorff unit space and a base of compact open bisections.
\end{definition}

If $\G$ is an ample groupoid, the notation $B^{\rm co}(\G)$ \index{$B^{\rm co}(\G)$, $\mathcal{B}(\G^{(0)})$} stands for the set of all nonempty compact open bisections in $\G$, and $\mathcal{B}(\G^{(0)})$ stands for the set of nonempty compact open subsets of $\G^{(0)}$.

Recall that a topological space is said to be \textit{totally disconnected} if the only nonempty connected subsets are singletons, and \textit{$0$-dimensional} if every point has a neighbourhood base of clopen (i.e., closed and open) sets. These two notions are equivalent if the space is locally compact and Hausdorff \cite[Theorems 29.5 \& 29.7]{willard1970general}. The following proposition is similar to \cite[Proposition 4.1]{exel2010reconstructing}. It is useful for reconciling slightly different definitions in the literature (e.g., \cite{clark2014groupoid}) and for checking when an \'etale groupoid is ample.

\begin{proposition} \label{tot disc} 
	Let $\G$ be an \'etale groupoid such that $\G^{(0)}$ is Hausdorff. Then the following are equivalent:
	\begin{enumerate}[\rm (1)]
		\item \label{tot disc 1} $\G$ is an ample groupoid;
		\item \label{tot disc 2} $\G^{(0)}$ is locally compact and totally disconnected;
		\item \label{tot disc 3} Every open bisection is locally compact and totally disconnected.
	\end{enumerate}
\end{proposition}

\begin{proof}
	$(\ref{tot disc 1}) \Rightarrow (\ref{tot disc 2})$ Let $U \subseteq \G^{(0)}$ be open. Since $\G$ is ample and $\G^{(0)}$ is open, for every $x \in U$ there is a compact open bisection $B$ such that $x \in B \subseteq U \subseteq \G^{(0)}$. Moreover, $\G^{(0)}$ is Hausdorff, so $B$ is closed. This shows that $\G^{(0)}$ is locally compact and 0-dimensional (hence totally disconnected).
	
	$(\ref{tot disc 2}) \Rightarrow (\ref{tot disc 3})$ Every open bisection is homeomorphic to an open subspace of $\G^{(0)}$, so it is totally disconnected and locally compact.
	
	$(\ref{tot disc 3}) \Rightarrow (\ref{tot disc 1})$ Let $U$ be open in $\G$, and $x \in U$. Since $\G$ is \'etale, it has a base of open bisections, so there is an open bisection $B$ with $x \in B \subseteq U$. Moreover, $B$ is Hausdorff, locally compact, and totally disconnected, so $x$ has a compact neighbourhood $W \subseteq B$ and a clopen neighbourhood $V \subseteq W$. Since $B$ is Hausdorff and $V$ is closed in $W$, it follows that $V$ is compact. Moreover, $V$ is an open bisection because $B$ is an open bisection. So, $V$ is a compact open bisection. This shows that $\G$ has a base of compact open bisections, so $\G$ is ample.
\end{proof}


\begin{remark} \label{subgroupoid}
	If $\G$ is a topological groupoid and $\mathcal{E}$ is a subgroupoid of $\G$, then $\mathcal{E}$ is automatically a topological groupoid with the topology it inherits from $\G$. 
	If $\G$ is \'etale, then so is $\mathcal{E}$.
	However, if $\G$ is ample, then it is not guaranteed that $\mathcal{E}$ is ample. Indeed, by Proposition \ref{tot disc}~(\ref{tot disc 2}),  a subgroupoid $\mathcal{E}$ of an ample groupoid $\G$ is ample if and only if $\mathcal{E}^{(0)}$ is locally compact. In particular, $\mathcal{E}$ is ample if $\G$ is ample and $\mathcal{E}^{(0)}$ is either open or closed in $\mathcal{G}^{(0)}$.
\end{remark}

The following lemma is similar to \cite[Proposition 2.2.4]{paterson1999groupoids}, but with slightly different assumptions.

\begin{lemma} \label{invsemigroup}
	Let $\G$ be an \'etale groupoid where $\G^{(0)}$ is Hausdorff. If $A,B,C \subseteq \G$ are compact open bisections, then
	\begin{enumerate}[\rm (1)]
		\item \label{invsemigroup 1} $A^{-1} = \{a^{-1} \mid a \in A\}$ and $AB = \{ab \mid (a, b) \in 
		(A \times B) \cap\G^{(2)} \}$ are compact open bisections.
		\item \label{invsemigroup 2} If $\G$ is Hausdorff, then $A \cap B$ is a compact open bisection.
	\end{enumerate}
\end{lemma}

\begin{proof}
	(\ref{invsemigroup 1}) Firstly, $A^{-1} = \inv(A)$ is compact and open because $\inv$ is a homeomorphism. Clearly, $A^{-1}$ is an open bisection. Secondly, note that $AB$ might be empty, in which case it is trivially a compact open bisection. Otherwise, $(A\times B) \cap \G^{(2)}$ is compact because $\G^{(2)}$ is closed in $\G \times \G$, and $AB = \mult\big((A \times B) \cap \G^{(2)}\big)$ is compact because $\mult$ is continuous. Since $\mult$ is a local homeomorphism, it is an open map, and $AB = \mult\big((A \times B) \cap \G^{(2)}\big)$ is open. To prove that it is a bisection, suppose $(a, b)$ is a composable pair in $A \times B$ and $\dom(ab) = x$. Since $A$ and $B$ are bisections, $b$ is the unique element in $B$ having $\dom(b) = x$, and $a$ is the unique element of $A$ having $\dom(a) = \cod(b)$. So, $\dom{|_{AB}}$ is injective. Similarly, $\cod{|_{AB}}$ is injective.
	
	(\ref{invsemigroup 2}) It is trivial that $A\cap B$ is an open bisection. The Hausdorff property on $\G$ implies $A$ and $B$ are closed, so $A \cap B$ is closed, hence compact.
\end{proof}

Lemma \ref{invsemigroup} remains true if the words ``compact" or ``open", or both, are removed throughout the statement. Using Lemma \ref{invsemigroup}~(\ref{invsemigroup 2}) with mathematical induction shows that when an ample groupoid is Hausdorff, its set of compact open bisections is closed under finite intersections. The converse to this statement is also true: an ample groupoid is Hausdorff  if the set of compact open bisections is closed under finite intersections (see \cite[Proposition 3.7]{steinberg2010groupoid}).

The main takeaway from Lemma \ref{invsemigroup}~(\ref{invsemigroup 1}) is that the compact open bisections in an ample groupoid are important for two reasons: they generate the topology, and they can be multiplied and inverted in a way that is consistent with an algebraic structure called an inverse semigroup.
An \textit{inverse semigroup} is a semigroup $S$ such that every $s \in S$ has a unique \textit{inverse} $s^*\in S$ with the property $ss^*s = s$ and $s^*ss^* = s^*$.

\begin{example}
	If $X$ is a set, a \textit{partial symmetry} of $X$ is a bijection $s: \domm(s) \to \codd(s)$ where $\domm(s)$ and $\codd(s)$ are (possibly empty) subsets of $X$. Two partial symmetries $s$ and $t$ are composed in the way that binary relations are composed, so that $st: \domm(st) \to \codd(st)$ is the map $st(x) = s(t(x))$ for all $x \in X$ such that $s(t(x))$ makes sense. It is \textit{not} necessary to have $\domm(s) = \codd(t)$ in order to compose $s$ and $t$. The semigroup $\mathcal{I}_X$ of partial symmetries on $X$ is called the \textbf{symmetric inverse semigroup} on $X$. The Wagner-Preston Theorem is an analogue of Cayley's Theorem for groups: every inverse semigroup $S$ has an embedding into $\mathcal{I}_S$.
\end{example}

The following result is an adaptation of \cite[Proposition 2.2.3]{paterson1999groupoids}.

\begin{proposition} \label{inv-semigroup prop}
	If $\G$ is an ample groupoid, $B^{\rm co}(\G)$ is an inverse semigroup with the inversion and composition rules displayed in Lemma \ref{invsemigroup}~(\ref{invsemigroup 1}).
\end{proposition}

\begin{proof}
	Lemma \ref{invsemigroup}~(\ref{invsemigroup 1}) proves that $B^{\rm co}(\G)$ is a semigroup and that $A \in B^{\rm co}(\G)$ implies $A^{-1} \in B^{\rm co}(\G)$. If $A \in B^{\rm co}(\G)$ then $AA^{-1} = \cod(A)$ because all composable pairs in $A \times A^{-1}$ are of the form $(a,a^{-1})$ for some $a \in A$. Therefore $AA^{-1}A = \cod(A)A = A$ and $A^{-1}AA^{-1} = A^{-1}\cod(A) = A^{-1}\dom(A^{-1}) = A^{-1}$. To show that the inverses are unique, suppose $B \in B^{\rm co}(\G)$ satisfies $ABA = A$ and $BAB = B$. Then for all $a \in A$ there exists $b \in B$ such that $aba = a$. But then $b = a^{-1}aa^{-1} = a^{-1}$. This shows $A^{-1} \subseteq B$. Similarly, $BAB=B$ implies $B^{-1} \subseteq A$ and consequently $B \subseteq A^{-1}$. Therefore $B = A^{-1}$. 
\end{proof}

The proposition above has shown how to associate an inverse semigroup to an ample groupoid. The connections between ample groupoids and inverse semigroups run much deeper than this. There are at least two ways to associate an ample groupoid $\G$ to an inverse semigroup $S$. The first is the \textit{underlying groupoid} $\G_S$, where the underlying set is $S$, the topology is discrete, the unit space is the set of idempotents in $S$, and $\dom(s) = s^*s$ while $\cod(s) = ss^*$, for every $s \in S$. Composition in $\G_S$ is the binary operation from $S$, just restricted to composable pairs. The second way to associate an ample groupoid to an inverse semigroup $S$ is more complicated. It is called the \textit{universal groupoid} of $S$, and it only differs from the underlying groupoid when $S$ is large (i.e., fails to have some finiteness conditions). The universal groupoid has a topology that makes it ample but not necessarily Hausdorff. The universal groupoid of $S$ is quite powerful (as shown in \cite{steinberg2010groupoid}) because its Steinberg algebra $A_R(\G(S))$ is isomorphic to the inverse semigroup algebra $RS$. This takes us beyond our scope and, after all, we still need to define Steinberg algebras.

\subsection{Introducing Steinberg algebras} \label{stein alg}
The purpose of this section is to define and characterise the Steinberg algebra of an ample groupoid over a unital commutative ring $R$. 
Throughout this section, assume $\G$ is an ample groupoid.
In order to make sense of continuity for $R$-valued functions, assume $R$ has the discrete topology.
The \textit{support} of a function $f: X \to R$ is defined as the set $\supp f  = \{x \in X \mid f(x) \ne 0\}$.\index{$\supp$} When $X$ has a topology, we say that $f$ is \textit{compactly supported} if $\supp f $ is compact. If every point $x \in X$ has an open neighbourhood $N$ such that $f{|_{N}}$ is constant, then $f$ is called \textit{locally constant}. It is easy to prove that $f: X \to R$ is locally constant if and only if it is continuous. We use the following notation for the \textit{characteristic function} of a subset $U$ of $\G$:
\begin{align*} \index{$\bm{1}_U$}
&\bm{1}_U: \G \to R; && \bm{1}_U(g)=
\begin{cases}
1 & \text{if } g \in U \\
0 & \text{if } g \notin U.
\end{cases}
\end{align*}
Let $R^\G$ be the set of all functions $f: \G \to R$. Canonically, $R^\G$ has the structure of an $R$-module with operations defined pointwise.

\begin{definition}[The Steinberg algebra] \index{$A_R(\G)$}
	Let $A_R(\G)$ be the $R$-submodule of $R^\G$ generated by the set:
	\begin{equation*} \label{The Definition}
	\{\bm{1}_U  \mid U \text{ is a Hausdorff compact open subset of }\G\}.
	\end{equation*}
	The \textit{convolution} of $f, g \in A_R(\G)$ is defined as 
	\begin{align} \label{convolution}
	f*g(x)= \sum_{\substack{y \in \G\\ \dom(y) = \dom(x)}}f(xy^{-1})g(y) = \sum_{\substack{(z, y)\in \G^{(2)} \\ zy  = x}} f(z)g(y) && \text{for all } x \in \G.
	\end{align}
	The $R$-module $A_R(\G)$,  with the convolution, is called the \textbf{Steinberg algebra} of $\G$ over $R$.
\end{definition}

\begin{example}
	If $\Gamma$ is a discrete group, then $A_R(\Gamma)$ is isomorphic to $R\Gamma$, the usual \textbf{group algebra} of $\Gamma$ with coefficients in $R$.
\end{example}

We have yet to justify the definition of the convolution in (\ref{convolution}). The two sums in the formula are equal, by substituting $z = xy^{-1}$. But it should not be taken for granted that the sum is finite, that $*$ is associative, or even that $A_R(\G)$ is closed under~$*$. These facts will be proved later. First, we prove the following result (inspired by \cite{steinberg2010groupoid}) that leads to some alternative descriptions of $A_R(\G)$ as an $R$-module.

\begin{proposition} \label{spann}
	Let $\mathcal{B}$ be a base for $\G$ consisting of Hausdorff compact open sets, with the property:
	\begin{equation*} \label{intersections}
	\left\{\bigcap_{i = 1}^n B_i \mid B_i \in \mathcal{B},\ \bigcup_{i = 1}^n B_i \emph{ is Hausdorff }\right\} \subseteq \mathcal{B} \cup \{\emptyset \}.
	\end{equation*}
	Then $A_R(\G) = \Span_R\{\bm{1}_B\mid B \in \mathcal{B}\}$.
\end{proposition}
\begin{proof}
	Let $A = \Span_R\{\bm{1}_B \mid B \in \mathcal{B}\}$. From the definition of $A_R(\G)$, we have $A \subseteq A_R(\G)$. To prove the other containment, suppose $U$ is a Hausdorff compact open subset of $\G$. It is sufficient to prove that $\bm{1}_U$ is an $R$-linear combination of finitely many $\bm{1}_{B_i}$, where each $B_i \in \mathcal{B}$. Since $\mathcal{B}$ is a base for the topology on $\G$, we can write $U$ as a union of sets in $\mathcal{B}$, and use the compactness of $U$ to reduce it to a finite union $U = B_1 \cup \dots \cup B_n$, where $B_1, \dots, B_n \in \mathcal{B}$. By the principle of inclusion-exclusion:
	\[
	\bm{1}_U = \sum_{k = 1}^n (-1)^{k-1} \sum_{\substack{I \subseteq \{1, \dots, n\} \\ |I|=k}}\bm{1}_{\cap_{i \in I}B_i}.
	\]
	The main assumption ensures that the sets $\cap_{i \in I}B_i$ on the right hand side are either empty or in $\mathcal{B}$. Therefore $A_R(\G) \subseteq A $.
\end{proof}

\begin{corollary} \label{cos span}
	If $\G$ is Hausdorff and $\mathcal{B}$ is a base of compact open sets that is closed under finite intersections, then $A_R(\G) = \Span_{R}\{\bm{1}_B  \mid  B \in \mathcal{B} \}$.
\end{corollary}


We remarked after Lemma \ref{invsemigroup} that if $\G$ is non-Hausdorff, $B^{\rm co}(\G)$ is \textit{not} closed under finite intersections. Strange things can happen in non-Hausdorff spaces and the problem lies in the fact that compact sets are not always closed, and the intersection of two compact sets is not always compact. However, $B^{\rm co}(\G)$ does satisfy the hypothesis of Proposition \ref{spann}.

\begin{corollary} \label{cob span}
	\cite[Proposition 4.3]{steinberg2010groupoid} \label{bco generates} The Steinberg algebra is generated as an $R$-module by characteristic functions of compact open bisections. That is,
	$
	A_R(\G) = \Span_R\{\bm{1}_B \mid B \in B^{\rm co}(\G)\}$.
\end{corollary}

\begin{proof}
	If $B_1, \dots, B_n \in B^{\rm co}(\G)$, and $U = \cup_i B_i$ is Hausdorff, then each $B_i$ is closed in $U$ because $U$ is compact, so $\cap_i B_i$ is closed in $U$. And, $B_1$ is a compact set containing the closed set $\cap_i B_i$, so $\cap_i B_i$ is compact. Clearly $\cap_i B_i$ is an open bisection, so $\cap_i B_i \in B^{\rm co}(\G)$.
\end{proof}

\begin{remark} \label{subgroupoid2}
	If $\mathcal{G}$ is an ample groupoid and $\mathcal{E}$ is an open subgroupoid, then $\mathcal{E}$ is also ample (see Remark \ref{subgroupoid}). Let $\iota: \mathcal{E} \hookrightarrow \mathcal{G}$ be the inclusion homomorphism. There is a canonical monomorphism $m: A_R(\mathcal{E}) \hookrightarrow A_R(\G)$, linearly extended from $\bm{1}_U \mapsto \bm{1}_{\iota(U)}$ for every Hausdorff compact open set $U \subseteq \mathcal{E}$.
	If $\mathcal{E}$ is closed, $m$ has a left inverse $e: A_R(\G) \twoheadrightarrow A_R(\mathcal{E})$,  linearly extended from $\bm{1}_U \mapsto \bm{1}_{U \cap \mathcal{E}}$ for every Hausdorff compact open set $U \subseteq \G$.
\end{remark}

We still owe a proof that the convolution, from equation (\ref{convolution}), is well-defined and gives an $R$-algebra structure to $A_R(\G)$. The next two results are similar to \cite[Propositions 4.5 \& 4.6]{steinberg2010groupoid}.

\begin{lemma} \label{char conv}
	Let $A, B, C \in B^{\rm co}(\G)$ and $r,s \in R$. Then:
	\begin{enumerate}[{\rm (1)}]
		\item \label{char conv 1}
		$\bm{1}_{A^{-1}}(x) = \bm{1}_{A}(x^{-1})$ for all $x \in \G$;
		\item \label{char conv 2}
		$\bm{1}_A * \bm{1}_B = \bm{1}_{AB}$;
	\end{enumerate}
\end{lemma}
\begin{proof}
	(\ref{char conv 1}) We have $x \in A^{-1}$ if and only if $x^{-1} \in A$.
	
	(\ref{char conv 2}) Let $x \in \G$. By definition:
	\begin{equation}\label{char conv eq1}
	\bm{1}_A*\bm{1}_{B}(x) = \sum_{\substack{y \in \G \\ \dom(y) = \dom(x)}}\bm{1}_A(xy^{-1})\bm{1}_B(y) = \sum_{\substack{y \in B \\ \dom(y) = \dom(x)}}\bm{1}_A(xy^{-1})
	\end{equation}
	Assume $x$ is of the form $x = ab$ where $a \in A$ and $b \in B$. Since $B$ is a bisection, $b$ is the only element of $B$ having $\dom(b) = \dom(x)$, and it follows that
	\[\ \bm{1}_A * \bm{1}_B (x) =  \bm{1}_A(xb^{-1}) = \bm{1}_A(a) = 1.\]
	On the other hand, assume $x \notin AB$. If there is $y \in B$ such that $\dom(y) = \dom(x)$, then $xy^{-1} \notin A$, for if it were, then $xy^{-1}y = x$ would be in $AB$. Therefore (\ref{char conv eq1}) yields $\bm{1}_A * \bm{1}_B (x) = 0$.
\end{proof}

Lemma \ref{char conv}~(\ref{char conv 2}) implies that characteristic functions of compact open subsets of the unit space can be multiplied pointwise. That is, if $V, W \in \mathcal{B}(\G^{(0)})$ then $VW = V\cap W = WV$ and $\bm{1}_V * \bm{1}_W (x) = \bm{1}_V(x)\bm{1}_W(x)$  for all $x \in \G$. As $\G^{(0)}$ is open in any ample groupoid $\G$, by Remark \ref{subgroupoid2}, there is a commutative subalgebra $A_R(\G^{(0)}) \hookrightarrow A_R(\G)$.

The ingredients of an \textit{$R$-algebra} are an $R$-module $A$ and a binary operation $A \times A \to A$. The binary operation should be $R$-linear in the first and second arguments (that is, bilinear), and it should be associative. There does not need to be a multiplicative identity. It is tedious to prove that $*$ is associative from its definition in (\ref{convolution}), so a proof was omitted in \cite{steinberg2010groupoid}.

\begin{proposition}
	The $R$-module $A_R(\G)$, equipped with the convolution, is an $R$-algebra.
\end{proposition}

\begin{proof}
	We need to show that the image of $*: A_R(\G) \times A_R(\G) \to R^\G$ is contained in $A_R(\G)$, and that $*$ is associative and bilinear. Bilinearity can be proved quite easily from formula (\ref{convolution}).
	Recall from Corollary \ref{cob span} that the elements of $A_R(\G)$ are $R$-linear combinations of characteristic functions of compact open bisections. If $f = \sum_{i} a_i \bm{1}_{A_i}$, $g = \sum_{j} b_j \bm{1}_{B_j}$, and $h = \sum_{k}c_k \bm{1}_{C_k}$, where the sums are finite, and  $A_i, B_j, C_k \in B^{\rm co}(\G)$ while $a_i, b_j, c_k \in R$ for all $i,j,k$, then 
	\[
	(f*g)*h = \sum_i \sum_j \sum_k a_i b_j c_k \bm{1}_{(A_i B_j)C_k} = \sum_i \sum_j \sum_k a_i b_j c_k \bm{1}_{A_i(B_j C_k)} = f*(g*h),
	\]
	using Lemma \ref{char conv}~(\ref{char conv 2}) and the bilinearity of $*$. This proves $*$ is associative. Evidently, $f*g = \sum_{i,j} a_i b_j \bm{1}_{A_iB_j} \in A_R(\G)$, so $A_R(\G)$ is closed under $*$.
\end{proof}

It is often useful to think of $*$ simply as the extension of the rule $\bm{1}_A * \bm{1}_B = \bm{1}_{AB}$ for all pairs $A,B \in B^{\rm co}(\G)$, rather than the more complicated-looking expression (\ref{convolution}) that we first defined it with. Moreover, one can infer from it that $A_R(\G)$ is a homomorphic image of the semigroup algebra of $B^{\rm co}(\G)$ with coefficients in $R$.

\begin{proposition} \label{ARG char}
	If $\G$ is Hausdorff and ample, then
	\begin{equation}\label{cont functions}
	A_R(\G) = \{f: \G \to R \mid f \text{ is locally constant, compactly supported}\big\}.
	\end{equation}
	Moreover, if $\mathcal{B}$ is a base for $\G$ consisting of compact open sets, such that $\mathcal{B}$ is closed under finite intersections and relative complements, then every nonzero $f \in A_R(\G)$ is of the form $f = \sum_{i = 1}^m r_i \bm{1}_{B_i}$, where $r_1, \dots, r_n \in R \setminus \{0\}$ and $B_1, \dots, B_n \in \mathcal{B}$ are mutually disjoint.
\end{proposition}

\begin{proof}
	Let $A$ be the set of locally constant, compactly supported $R$-valued functions on $\G$. Let $\mathcal{B}$ be a base of compact open sets for $\G$, such that $\mathcal{B}$ is closed under finite intersections and relative complements. (A worthy candidate for $\mathcal{B}$ is $B^{\rm co}(\G)$.) If $0 \ne f \in A_R(\G)$ then according to Corollary \ref{cos span}, $f = \sum_{i = 1}^n s_i \bm{1}_{D_i}$ for some basic open sets $D_i \in \mathcal{B}$ and nonzero scalars $s_i \in R$. We aim to rewrite it as a linear combination of characteristic functions of \textit{disjoint} open sets. If $s \in \im f \setminus \{0\}$, then we have the expression:
	\begin{align} \label{beast}
	f^{-1}(s) &= \bigcup_{\substack{I \subseteq \{1, \dots, n\} \\ s = \sum_{i \in I}{s_i}}}\  B_I, & \text{where } \qquad B_I &= \bigcap_{\substack{i\in I \\ j \notin I}} D_i \setminus D_j.
	\end{align}
	By assumption, each nonempty $B_I$ in the expression is an element of $\mathcal{B}$; in particular, each $B_I$ is compact and open. Finite unions preserve openness and compactness, so $f^{-1}(s)$ is open and compact for every nonzero $s \in \im f$. It follows that $f^{-1}(0) = \G \setminus \left(\bigcup_{s \in \im f \setminus \{0\}} f^{-1}(s)\right)$ is open. Therefore $f$ is locally constant. As $f$ is a linear combination of $n$ characteristic functions, it is clear that $|\im f\setminus\{0\}| \le 2^n$. Being a finite union of compact sets, $\supp f  = \bigcup_{s \in \im f \setminus \{0\}}f^{-1}(s)$ is compact. Thus $f \in A$, and this shows $A_R(\G) \subseteq A$.
	To prove the other containment, that $A \subseteq A_R(\G)$, suppose $f \in A$. As $f$ is continuous and $\supp f $ is compact, $f(\supp f) = \im f \setminus \{0\}$ is compact in $R$, so it must be finite. Let $\im f \setminus \{0\} = \{r_1, \dots, r_n\}$. Then each set $U_i = f^{-1}(r_i)$ is clopen because $f$ is continuous, and compact because $U_i \subseteq \supp f $. Hence $f = \sum_{i = 1}^n r_i \bm{1}_{U_i} \in A_R(\G)$, and this shows $A \subseteq A_R(\G)$.
	
	To prove the ``moreover" part, we look again at (\ref{beast}). 	If $I, J \subseteq \{1, \dots, n\}$ and $I \ne J$ then $B_I \cap B_J = \emptyset$. Therefore, $f \in A_R(\G)$ can be written as an $R$-linear combination of characteristic functions of disjoint basic open sets in $\mathcal{B}$:
	\phantom\qedhere
	\[
	\pushQED{\qed}
	f = \sum_{s \in \im f \setminus \{0\}} s \bm{1}_{f^{-1}(s)} = \sum_{s \in \im f \setminus \{0\}} \sum_{\substack{I \subseteq \{1, \dots, n\} \\ s = \sum_{i \in I}{s_i}}} s \bm{1}_{B_I}.
	\qedhere
	\popQED
	\]
\end{proof}

\subsection{Properties of Steinberg algebras} \label{properties}

It is useful to know when $A_R(\G)$ is unital or has some property that is nearly as good. The answer is quite easy, and we show it below. We use the definition that a ring (or $R$-algebra) $A$ is \textit{locally unital} if there is a set of commuting idempotents $E \subseteq A$, called \textit{local units}, with the property: for every finite subset $\{a_1, \dots, a_n\} \subseteq A$, there is a local unit $e \in E$ with $e a_i = a_i = a_i e$ for every $1 \le i \le n$. Equivalently, $A$ is the direct limit of unital subrings: $A =\underset{\longrightarrow}{\lim_{e \in E}}\,eAe$. The directed system is facilitated by the partial order, $e \le e'$ if $ee' = e = e'e$, and the connecting homomorphisms (which need not be unit-preserving) are the inclusions $eAe \hookrightarrow e'Ae'$ for $e \le e'$.

In many respects, working with locally unital rings is like working with unital rings. Every locally unital ring $A$ is idempotent (i.e., $A^2 = A$) and if $I \subseteq A$ is an ideal, then $AI = I = IA$. If $A$ is an $R$-algebra with local units, then the ring ideals of $A$ are always $R$-algebra ideals (which, by definition, should be $R$-submodules of $A$). These facts are not true in general for arbitrary non-unital rings. Locally unital rings and algebras are always \textit{homologically unital}, in the sense of \cite[Definition 1.4.6]{loday}, which essentially means that they have well-behaved homology. The classical Morita Theorems, with slight adjustments, are valid for rings with local units (see \cite{anh1987morita}).

\begin{proposition} \cite[Proposition 4.11]{steinberg2010groupoid} \label{units}, \cite[Lemma 2.6]{clark2018generalized}.
	Let $\G$ be an ample groupoid. Then $A_R(\G)$ is locally unital. Moreover, $A_R(\G)$ is unital if and only if $\G^{(0)}$ is compact.
\end{proposition}

\begin{proof}
	We prove the ``moreover" part first. If $\G^{(0)}$ is compact, then it is a compact open bisection, and $\bm{1}_{\G^{(0)}} \in A_R(\G)$. Following Lemma \ref{char conv}~(\ref{char conv 2}), $\bm{1}_{\G^{(0)}}*\bm{1}_{B} = \bm{1}_{\G^{(0)}B} = \bm{1}_B =  \bm{1}_{B \G^{(0)}} = \bm{1}_{B}*\bm{1}_{\G^{(0)}}$, for every $B \in B^{\rm co}(\G)$. Since $\{\bm{1}_B \mid B \in B^{\rm co}(\G)\}$ spans $A_R(\G)$, it follows by linearity that $\bm{1}_{\G^{(0)}} * f = f = f * \bm{1}_{\G^{(0)}}$ for every $f \in A_R(\G)$. This proves that $\bm{1}_{\G^{(0)}}$ is the multiplicative identity in $A_R(\G)$.
	
	Conversely, suppose $A_R(\G)$ has a multiplicative identity called $\xi$. The first step is to show that $\xi = \bm{1}_{\G^{(0)}}$. Let $x \in \G$ and let $V \subseteq \G^{(0)}$ be a compact open set containing $\dom(x)$. Then $V$ must be Hausdorff because $\G^{(0)}$ is, so $\bm{1}_V \in A_R(\G)$. If $x \notin \G^{(0)}$, then
	\[
	0 = \bm{1}_V(x) = \xi* \bm{1}_V (x) =  \sum_{y \in  \G{\dom(x)}}\xi(xy^{-1})\bm{1}_V(y) = \sum_{y \in V \cap \G{\dom(x)}} \xi(xy^{-1}) = \xi(x)
	\]
	because $V \cap \G{\dom(x)}  = \{\dom(x)\}$.
	Similarly, if $x \in \G^{(0)}$ then $x = \dom(x) \in V$ and
	\[
	1 = \bm{1}_V(x) = \xi* \bm{1}_V(x) = \xi(x).
	\]
	This shows that $\xi = \bm{1}_{\G^{(0)}}$. The second step is to show that $\bm{1}_{\G^{(0)}} \in A_R(\G)$ implies $\G^{(0)}$ is compact. By the definition of $A_R(\G)$, there exist scalars $r_1, \dots, r_n \in R \setminus \{0\}$ and compact open sets $U_1, \dots, U_n \subseteq \G$ such that $\bm{1}_{\G^{(0)}} = r_1 \bm{1}_{U_1} + \dots + r_n\bm{1}_{U_n}$. Then $\G^{(0)} \subseteq U_1 \cup \dots \cup U_n$ and consequently $\G^{(0)} = \dom(U_1) \cup \dots \cup \dom(U_n)$. Each of the sets $\dom(U_1), \dots, \dom(U_n)$ is compact (because $\dom$ is continuous), so $\G^{(0)}$ is compact.
	
	To show that $A_R(\G)$ is locally unital for all ample groupoids $\G$, suppose $F = \{f_1, \dots, f_m\}$ is a finite subset of $A_R(\G)$. Since $A_R(\G)$ is spanned by $\{\bm{1}_B \mid B \in B^{\rm co}(\G)\}$, there exist finite subsets $\{B_1, \dots, B_n\} \subseteq B^{\rm co}(\G)$ and $\{r_{i,j} \mid 1 \le i \le n, 1 \le j \le m \}\subseteq R$ such that $f_j = r_{1,j}\bm{1}_{B_1} + \dots + r_{n,j}\bm{1}_{B_n}$ for all $1 \le j \le m$. Let $X = \dom(B_1) \cup \dots \cup \dom(B_n) \cup \cod(B_1) \cup \dots \cup \cod(B_n)$. Then $X \subseteq G^{(0)}$ is compact and open because it is a finite union of compact open sets, and $X$ is Hausdorff because it is a subset of $\G^{(0)}$, so $\bm{1}_X \in A_R(\G)$. Clearly, $XB_i = B_i = B_iX$, so $\bm{1}_X *\bm{1}_{B_i} = \bm{1}_{B_i} = \bm{1}_{B_i}* \bm{1}_X$, for all $1 \le i \le n$. By linearity, $\bm{1}_X * f_j = f_j = f_j* \bm{1}_X$ for all $1 \le j \le m$. The conclusion is that $E = \{\bm{1}_X \mid X \in \mathcal{B}(\G^{(0)})\}$ is a set of local units for $A_R(\G)$.
\end{proof}


%

The \textit{characteristic} of a ring $A$, written $\Char A$, is defined as the least positive integer $n$ such that $n \cdot a = 0$ for all $a \in A$, or 0 if no such $n$ exists. If $A$ has a set of local units $E$, the characteristic of $A$ can be defined as the least $n$ such that $n \cdot e = 0$ for all $e \in E$, or 0 if no such $n$ exists. 

\begin{proposition}
	For any ample groupoid $\G$, $\Char A_R(\G) = \Char R$.
\end{proposition}
\begin{proof}
	If $n$ is a positive integer, $n \cdot \bm{1}_U = 0$ for all $U \in \mathcal{B}(\G^{(0)})$ if and only if $n \cdot 1 = 0$.
\end{proof}

Given a topological groupoid $(\G, \G^{(0)}, \dom, \cod, \mult, \inv)$, the \textit{opposite groupoid} is:
\[
\G^{\rm op} =(\G, \G^{(0)}, \dom^{\rm op}, \cod^{\rm op}, \mult^{\rm op}, \inv)
\]
where $\dom^{\rm op} = \cod$, $\cod^{\rm op} = \dom$, and $\mult^{\rm op}(x,y) = \mult(y,x)$ for any $x, y$ with $\cod(x) = \dom(y)$. We call the opposite groupoid $\G^{\rm op}$ to distinguish it from $\G$, even though they have the same underlying sets. We assume $\G^{\rm op}$ has the same topology as $\G$. Naturally, the inversion map $\inv: \G \to \G^{\rm op}$ is an isomorphism of topological groupoids.

If $A$ is a ring, an \textit{involution} on $A$ is an additive, anti-multiplicative map $\tau: A \to A$ such that $\tau^2 = \id_A$. If $A$ has an involution, it is called an \textit{involutive ring} or \textit{$*$-ring}. If $\G$ is an ample groupoid, $f \mapsto f \circ \inv$ is a canonical involution on $A_R(\G)$ that makes it a $*$-algebra. More generally, if there is an involution $\overline{\phantom{x}}: R \to R$, written as $r \mapsto \overline{r}$, then $f \mapsto \overline{f \circ \inv}$ is an involution on $A_R(\G)$. To summarise:

\begin{proposition} \label{involution proposition}
	Let $\G$ be an ample groupoid. There are canonical isomorphisms $\G \cong \G^{\rm op}$ and $A_R(\G) \cong A_R(\G^{\rm op}) \cong A_R(\G)^{\rm op}$. Moreover, to each involution $\overline{\phantom{x}}: R \to R$ is associated a canonical involution on $A_R(\G)$, namely $f \mapsto \overline{f \circ \inv}$ for all $f \in A_R(\G)$.
\end{proposition}

This kind of symmetry is very nice to work with. It implies, for example, that the category of left $A_R(\G)$-modules is isomorphic to the category of right $A_R(\G)$-modules, and the lattice of left ideals in $A_R(\G)$ is isomorphic to the lattice of right ideals. Many important notions, like left and right primitivity, are equivalent for involutive algebras (or more generally, self-opposite algebras).

\subsection{First examples} \label{first examples}

One or two of the results in this section will be useful later on, but mostly they are just interesting in their own right. Presumably, most of this content is already known, but we do not adhere closely to any references.

Given two groupoids $(\G_1, \dom_1, \cod_1, \mult_1, \inv_1)$ and $(\G_2, \dom_2, \cod_2, \mult_2, \inv_2)$, their \textit{disjoint union} $\G_1 \sqcup \G_2$ has the structure of a groupoid with unit space $\G_1^{(0)} \sqcup \G_2^{(0)}$, set of composable pairs $\G_1^{(2)} \sqcup \G_2^{(2)}$, and the following structure maps: for all $x_1, y_1 \in \G_1$ and $x_2, y_2 \in \G_2$,
\begin{align*}
\dom(x_i) &= \dom_i(x_i), & \cod(x_i) &= \cod_i(x_i), & \inv(x_i) &= \inv_i(x_i), & \mult(x_i,y_i) &= \mult_i(x_i, y_i).
\end{align*}
The \textit{product} $\G_1 \times \G_2$ also has the structure of a groupoid with unit space $\G_1^{(0)}\times \G_2^{(0)}$, and the following structure maps: for all $x_1, y_1 \in \G_1$ and $x_2, y_2 \in \G_2$,
\begin{align*}
\dom(x_1, x_2) &= (\dom_1(x_1),\dom_2(x_2)), & \cod(x_1,x_2) &= (\cod_1(x_1),\cod_2(x_2)), \\
\inv(x_1, x_2) &= (\inv_1(x_1), \inv_2(x_2)), & \mult((x_1, x_2),(y_1, y_2)) &= (\mult_1(x_1,y_1),\mult_2(x_2,y_2)).
\end{align*}
These constructions work just as well for the disjoint union or product of arbitrarily many (even infinitely many) groupoids. If $\G_1$ and $\G_2$ are topological groupoids, then $\G_1 \sqcup \G_2$ (with the coproduct topology) and $\G_1 \times \G_2$ (with the product topology) are again topological groupoids. The properties of being \'etale or ample are preserved by arbitrary disjoint unions and finite products.

\begin{proposition} \label{coprod}
	Let $\G_1$ and $\G_2$ be ample groupoids. The Steinberg algebra of $\G_1 \sqcup \G_2$ is a direct sum of two ideals:
	$A_R(\G_1 \sqcup \G_2) \cong A_R(\G_1)\oplus A_R(\G_2)$.
\end{proposition}

\begin{proof}
	Let $I_1 = \{f_1 \in A_R(\G_1 \sqcup \G_2) \mid \supp f_1 \subseteq \G_1\}$ and $I_2 = \{f_2 \in A_R(\G_1 \sqcup \G_2)\mid \supp f_2 \subseteq \G_2\}$. Recall from Remark \ref{subgroupoid2} that $I_1 \cong A_R(\G_1)$ and $I_2 \cong A_R(\G_2)$.
	Every $f \in A_R(\G_1 \sqcup \G_2)$ decomposes as $f = f_1 + f_2$ where $f_i \in I_i$ are defined as:
	\[
	f_i(x) = \begin{cases} 
	f(x)  &\text{if } x \in \G_i\\
	0 & \text{if } x \notin \G_i
	\end{cases}
	\]
	for $i = 1,2$. We claim $I_1$ and $I_2$ are orthogonal ideals (that is, $I_1 * I_2 = 0$). For all $f_1 \in I_1$, $f_2 \in I_2$, and $x \in \G_1 \sqcup \G_2$, $f_1*f_2(x) = \sum_{ab = x}f_1(a)f_2(b)$. So, $\supp(f_1 * f_2) \subseteq \supp(f_1)\supp(f_2) \subseteq \G_1 \G_2 = \emptyset$. This implies $I_1$ and $I_2$ are ideals, and $A_R(\G_1 \sqcup \G_2) = I_1 \oplus I_2 \cong A_R(\G_1) \oplus A_R(\G_2)$.
\end{proof}

By mathematical induction, the Steinberg algebra of a finite disjoint union of ample groupoids is isomorphic to the direct sum of their respective Steinberg algebras.

Like in \cite[Notation 2.6.3]{LPAbook}, we have reasons to consider matrix rings of a slightly more general nature than usual.

\begin{definition}[Matrix rings] \index{$M_n(A)$, $M_\Lambda(A)$} Let $A$ be a ring (not necessarily commutative or unital). If $n$ is a positive integer, we write $M_n(A)$ for the ring of $n \times n$ matrices with entries in $A$. If $\Lambda$ is a set (not necessarily finite) we define $M_\Lambda(A)$ to be the ring of square matrices, with rows and columns indexed by $\Lambda$, having entries in $A$ and only finitely many nonzero entries.
\end{definition}

Note that $M_\Lambda(A)$ is the direct limit of the finite-sized matrix rings associated to finite subsets of $\Lambda$. Also, $M_\Lambda(A)$ is unital if and only if $A$ is unital and $\Lambda$ is finite. The notation $[a_{ij}]$ stands for the matrix in $M_n(A)$, or $M_\Lambda(A)$, with $a_{ij}$ in its $(i,j)$-entry.
Let $\mathcal{N} = \{1, \dots, n\}^2$ be the transitive principal groupoid on $n$ elements, with the discrete topology, as seen in Example \ref{groupoid ex}~(\ref{pairs}).

\begin{proposition} \label{matrix algebra}
	If $\G$ is a Hausdorff ample groupoid, then $A_R(\mathcal{N} \times \G) \cong M_n(A_R(\G)))$.
\end{proposition}

\begin{proof}
	Define the map $F: A_R(\mathcal{N} \times \G) \to M_n(A_R(\G))$:
	\begin{align*}
	&F(f) = [f_{ij}], &&\text{where } f_{ij}(x) = f\big((i,j),x\big)\text{ for all } f \in A_R(\mathcal{N} \times \G), (i,j) \in \mathcal{N}, \text{ and } x \in \G.
	\end{align*}
	If $f \in A_R(\mathcal{N} \times \G)$, then $f$ is compactly supported and locally constant. The restriction of $f$ to a clopen subset, such as $\{(i,j)\}\times \G$ for some $(i, j) \in \mathcal{N}$, is also compactly supported and locally constant. Therefore $f_{i,j} \in A_R(\G)$ for all $(i,j) \in \mathcal{N}$. Clearly, $F$ is bijective.
	Now, let $f, g \in A_R(\mathcal{N}\times \G)$. For all $(i,j) \in \mathcal{N}$ and $x \in \G$, the convolution formula yields
	\begin{align*}
	(f*g)_{ij}(x) = f*g \big((i,j),x\big)
	&= \sum_{\substack{(k,\ell,y)\in \mathcal{N} \times \G \\ (\ell,\dom(y)) = (j,\dom(x))}} f\big[((i,j),x)((k,\ell),y)^{-1}\big]g((k,\ell),y)\\
	&= \sum_{1 \le k \le N}\sum_{\substack{y \in \G \\ \dom(y)=\dom(x)}}
	f((i,k),xy^{-1})g((k,j),y) = \sum_{1 \le k \le n}f_{ik}* g_{kj}(x)
	\end{align*}
	This shows $F(f*g) =  F(f)F(g)$, so $F$ is an isomorphism.
\end{proof}

\begin{remark}
	As a specialisation of Proposition \ref{matrix algebra}, we obtain $A_R(\mathcal{N}) \cong M_n(R)$. It is well-known that when $A$ is an $R$-algebra, $M_n(A) \cong M_n(R) \otimes_R A$ (see \cite[Example 4.22]{brevsar2014introduction}). It is also well-known (see \cite[Example 4.20]{brevsar2014introduction}) that if $G$ and $H$ are groups, then $R(G \times H) \cong RG \otimes_R RH$. One can show using the standard techniques that when $\G_1$ and $\G_2$ are arbitrary ample groupoids, there is a surjective homomorphism $A_R(\G_1)\otimes_R A_R(\G_2) \to A_R(\G_1 \times \G_2)$. An interesting question is: under what circumstances is it an isomorphism?
\end{remark}

Suppose $\G$ is a topological groupoid and $\{\G_i\}_{i \in I}$ is a family of open subgroupoids indexed by a directed set $(I, \le)$, such that $\G = \bigcup_{i \in I} \G_i$ and $\G_i \subseteq \G_j$ whenever $i \le j$ in $I$. If this happens, we say that $\G$ is the \textit{directed union} of the subgroupoids $\{\G_i\}_{i \in I}$.

\begin{proposition} \label{directed union}
	If a Hausdorff ample groupoid $\G$ is the directed union of a family of open subgroupoids $\{\G_i\}_{i \in I}$, then $A_R(\G)$ is the direct limit of subalgebras $\{A_R(\G_i)\}_{i \in I}$.
\end{proposition}

\begin{proof}
	For all $i \le j$ in $I$, let $\varphi_{ij}: A_R(\G_i) \hookrightarrow A_R(\G_j)$ and $m_i: A_R(\G_i)\hookrightarrow A_R(\G)$ be the canonical embeddings (see Remark \ref{subgroupoid2}). We claim that for every $f \in A_R(\G)$, there exists $j \in I$ such that $f \in m_j(A_R(\G_j))$. If $f \in A_R(\G)$ then $\supp f$ is compact and open. Thus, there is a finite subcover of $\{\G_i\}_{i \in I}$ that covers $\supp f$. If $\supp f \subseteq \G_{i_1}\cup \dots \cup \G_{i_n}$, then there exists $j \in I$ with $i_1, \dots, i_n \le j$, using the fact that $(I, \le)$ is directed. Thus, $\supp f \subseteq \G_j$, and $f|_{\G_j}$ is compactly supported and locally constant, whereby $f|_{\G_j} \in A_R(\G_j)$. Finally, this shows $f = m_j(f|_{\G_j}) \in m_j(A_R(\G_j))$.
	
	Now assume $B$ is an $R$-algebra and $\{\beta_i\}_{i \in I}$ is a family of $R$-homomorphisms $\beta_i: A_R(\G_i) \to B$, such that $\beta_i = \beta_j \varphi_{ij}$ for all $i \le j$. Then, since every $\varphi_{ij}: A_R(\G_i) \to A_R(\G_j)$ is injective, $\beta_j$ is an extension of $\beta_i$ whenever $i \le j$. Since $A_R(\G) = \bigcup_{i \in I} m_i(A_R(\G_i))$, it follows that there is a unique homomorphism $\beta: A_R(\G) \to B$ such that $\beta_i = \beta m_i$ for all $i \in I$. As such, $A_R(\G)$ has the universal property for the directed system $\{A_R(\G_i)\}_{i \in I}$, so we can conclude it is the direct limit of that system.
\end{proof}

We can now extend Propositions \ref{coprod} and \ref{matrix algebra} to allow infinite index sets. This could have been proved directly, mentioning that the functions in $A_R(\G)$ have compact supports, but it is nice to demonstrate direct limits.

\begin{proposition} \label{infinities} Let $\G$ be a Hausdorff ample groupoid, and let $\Lambda$ be an infinite set.
	\begin{enumerate}[\rm (1)]
		\item \label{inf matrix} If $\mathcal{D} = \Lambda^2$ is the transitive principal groupoid on $\Lambda$, equipped with the discrete topology, then $A_R(\mathcal{D} \times \G) \cong M_\Lambda(A_R(\G))$.
		\item \label{inf coprod} If $\G = \bigsqcup_{\lambda \in \Lambda}\G_\lambda$ is the disjoint union of an infinite family of clopen subgroupoids $\{\G_\lambda\}_{\lambda \in \Lambda}$, then $A_R(\G) \cong \bigoplus_{\lambda \in \Lambda} A_R(\G_\lambda)$.
	\end{enumerate}
\end{proposition}

\begin{proof}
	(\ref{inf matrix}) Note that $\mathcal{D} \times \G$ is the directed union of the subgroupoids $\mathcal{D}_F \times \G$, where $\mathcal{D}_F = \{(d_1, d_2) \in \mathcal{D} \mid d_1, d_2 \in F\}$, as $F$ ranges over all the finite subsets of $\Lambda$ ordered by inclusion. By Propositions \ref{matrix algebra} and \ref{directed union}, $A_R(\mathcal{D} \times \G)$ is the direct limit of matrix algebras $A_R(\mathcal{D}_F \times \G) \cong M_F(A_R(\G))$, and this direct limit  is isomorphic to $M_\Lambda(A_R(\G))$.
	
	(\ref{inf coprod}) Note that $\G$ is the directed union of the subgroupoids $\G_F = \bigsqcup_{\lambda \in F}\G_\lambda$, as $F$ ranges over finite subsets of $\Lambda$ ordered by inclusion. By Propositions \ref{coprod} and \ref{directed union}, $A_R(\G)$ is the direct limit of the subalgebras $A_R(\G_F) \cong \bigoplus_{\lambda \in F}A_R(\G_\lambda)$, and this direct limit is isomorphic to $\bigoplus_{\lambda \in \Lambda} A_R(\G_\lambda)$.
\end{proof}

Here we describe a class of principal groupoids, called \textit{approximately finite} groupoids, that was defined by Renault in his influential monograph \cite{renault1980groupoid}.

\begin{example}  Let $X$ be a locally compact, totally disconnected Hausdorff space. Consider it as a groupoid with unit space $X$ and no morphisms outside the unit space. Then $A_R(X)$ is the commutative $R$-algebra of locally constant, compactly supported functions $f: X \to R$, with pointwise addition and multiplication. We adopt the notation $A_R(X) = C_R(X)$ and drop the $*$ notation for products, because this  serves as a reminder that $C_R(X)$ is commutative. \index{$C_R(X)$}
	An ample groupoid is called \textbf{elementary} if it is of the form $(\mathcal{N}_1\times X_1)\sqcup \dots \sqcup (\mathcal{N}_t \times X_t)$, where $\mathcal{N}_1, \dots \mathcal{N}_t$ are discrete, finite, transitive principal groupoids on $n_1, \dots, n_t$ elements, respectively, and $X_1, \dots, X_n$ are locally compact, totally disconnected, Hausdorff topological spaces. Using the results of this section:
	\begin{equation} \label{matricial}
	A_R\left(\bigsqcup_{i = 1}^n \left(\mathcal{N}_i \times X_i\right)\right) \cong \bigoplus_{i = 1}^t M_{n_i}\big(C_R(X_i)\big).
	\end{equation}
	A groupoid is called \textbf{approximately finite} if it is the directed union of an increasing sequence of elementary groupoids. The Steinberg algebra of an approximately finite groupoid is a direct limit of matricial algebras, each resembling (\ref{matricial}).
\end{example}

\begin{definition} \label{vN definition}
	A ring $A$ is called \textbf{von Neumann regular} if for every $x \in A$ there exists $y \in A$ such that $x = xyx$. 
\end{definition}

If $y \in A$ satisfies $x = xyx$ then $y$ is called a \textit{von Neumann inverse} of $x$. If $R$ is a commutative von Neumann regular ring, then for every $r \in R$ there exists a unique element $s \in R$ such that $r = r^2 s$ and $s = s^2 r$ (see \cite[Proposition 3.6]{goodearl}).

\begin{proposition} \label{vNr}
	If $\mathcal{F}$ is an approximately finite groupoid and $R$ is a von Neumann regular unital commutative ring, then $A_R(\mathcal{F})$ is von Neumann regular.
\end{proposition}

\begin{proof}
	Let $X$ be a locally compact, totally disconnected, Hausdorff topological space, and suppose $R$ is von Neumann regular. To verify that $C_R(X)$ is von Neumann regular, take $f \in C_R(X)$ and for every $x \in X$ define $g(x)$ to be the unique element of $R$ such that $f(x) = f(x)^2 g(x)$ and $g(x) = g(x)^2 f(x)$. Note that $g \in C_R(X)$ and $fgf = f$. Now, $C_R(X)$ being regular implies $M_n(C_R(X))$ is regular (this could be argued carefully with Morita equivalence, but one finds in \cite[Theorem 24]{kaplansky1972fields} a clever direct proof by induction). A direct sum of regular rings is regular, so any ring of the form (\ref{matricial}) is regular, provided $R$ is regular. A direct limit of regular rings is regular: each element in the direct limit must belong to a regular subring, and the von Neumann inverse can be chosen from that same subring. Therefore $A_R(\mathcal{F})$ is von Neumann regular.
\end{proof}

Note that we did not use the assumption that $\mathcal{F}$ is a countable directed union of elementary groupoids; any directed union will do.
It is an open problem to characterise von Neumann regularity for Steinberg algebras in groupoid terms; partial progress is achieved in \cite{ambily2018simple}.

This next result is a ``baby version" of \cite[Proposition 3.1]{steinberg2018chain}, with a new proof. In preparation for it, we briefly remark that every transitive groupoid $\G$ is (algebraically, but not necessarily topologically) isomorphic to the product of a transitive principal groupoid and a group. To construct such an isomorphism, fix a unit $b \in \G^{(0)}$. Let $\Gamma = { b\G b}$ be the isotropy group based at $b$, and let $\mathcal{P} = [\G^{(0)}]^2$ be the transitive principal groupoid on $\G^{(0)}$. Fix a morphism $h_{y} \in {{b}\G y}$ for every $y \in \G^{(0)}$, and define the groupoid isomorphisms:
\begin{align*}
F&: \G \to \mathcal{P} \times \Gamma,& &F(g) = \left(\big(\cod(g), \dom(g)\big), h_{\cod(g)} g h_{\dom(g)}^{-1}\right)&& \text{for all } g \in \G; \\
F^{-1}&: \mathcal{P} \times \Gamma \to \G,&  &F^{-1}\big((x,y), \gamma\big) = h_x^{-1} \gamma h_y & &\text{for all } x, y \in \G^{(0)}, \gamma \in \Gamma.
\end{align*}

\begin{proposition} \label{finite dim st}
	Let $\mathbb{K}$ be a field and $\G$ an ample groupoid. Then $A_{\mathbb{K}}(\G)$ is finite-dimensional if and only if $\G$ is finite and has the discrete topology. If $\mathcal{O}_1, \dots, \mathcal{O}_t$ are the orbits of $\G$, and $\Gamma_1, \dots, \Gamma_t$ are the corresponding isotropy groups, then
	\[
	A_{\mathbb{K}}(\G) \cong \bigoplus_{i = 1}^t M_{\mathcal{O}_i}(R\Gamma_i).
	\]
\end{proposition}

\begin{proof}
	First of all, if $\G$ is discrete, then $\dim_\mathbb{K} A_\mathbb{K}(\G) = |\G|$, because $\{\bm{1}_{\{g\}} \mid g \in \G\}$ is a basis for $A_\mathbb{K}(\G)$, by Corollary \ref{cos span}. Thus, $A_\mathbb{K}(\G)$ is finite-dimensional if $\G$ is finite and discrete. Conversely, suppose $A_\mathbb{K}(\G)$ is finite-dimensional, and let $\{f_1, \dots, f_n\}$ be a basis. The image of each $f_i$ is finite, so $|\im f_1 \cup \dots \cup \im f_n|$ is bounded by some $M < \infty$. If $|\G^{(0)}| > M^n$ then, by the pigeonhole principle, there exists $u \ne v$ in $\G^{(0)}$ such that $f_i(u)=f_i(v)$ for all $1 \le i \le n$, and thus $f(u) = f(v)$ for all $f \in A_\mathbb{K}(\G)$. But $\G^{(0)}$ is Hausdorff, locally compact, and totally disconnected, so there is a compact open subset $U \subseteq \G^{(0)}$ with $u \in U$ and $v \notin U$. Since $\G^{(0)}$ is open in $\G$, it follows that $U$ is a compact open bisection in $\G$, so $\bm{1}_U \in A_\mathbb{K}(\G)$. We arrive at a contradiction, because $\bm{1}_U(u) \ne \bm{1}_U(v)$. Therefore $|\G^{(0)}| \le M^n < \infty$. A finite Hausdorff space is discrete, so $\G^{(0)}$ is discrete. As $\G$ is \'etale, it must also be discrete. Thus $\dim_\mathbb{K} A_\mathbb{K}(\G) = n = |\G|$. Given that $\G$ is finite and discrete, it is isomorphic to a disjoint union of transitive groupoids (one for each orbit), each of which is isomorphic to the product of a transitive principal groupoid (with as many elements as the corresponding orbit), and a finite group (the isotropy group of that orbit). The expression giving the structure of $A_\mathbb{K}(\G)$ follows from Propositions \ref{coprod} and~\ref{matrix algebra}.
\end{proof}

\subsection{Graded groupoids and graded Steinberg algebras} \label{gr grp}

Just as the Steinberg algebra of a groupoid inherits an involution from the groupoid, so it can inherit a graded structure. Many well-studied examples of Steinberg algebras receive a canonical group-grading that comes from a grading on the groupoid itself. We first introduce the concepts and terminology of graded groupoids and graded algebras.

A standing assumption is that $\Gamma$ \index{$\Gamma$, $\varepsilon$} is a group with identity $\varepsilon$. A ring $A$ is called a \textit{$\Gamma$-graded ring} if it decomposes as a direct sum of additive subgroups $A = \bigoplus_{\gamma \in \Gamma} A_\gamma$ such that $A_\gamma A_\delta \subseteq A_{\gamma\delta}$ for every $\gamma, \delta \in \Gamma$. The meaning of $A_\gamma A_\delta$ is the additive subgroup generated by all products $a b$ where $a \in A_\gamma, b \in A_\delta$. The additive group $A_\gamma$ is called the \textit{$\gamma$-component} of $A$. 
The elements of $\bigcup_{\gamma \in \Gamma} A_{\gamma}$ in a graded ring $A$ are called \emph{homogeneous elements}. The nonzero elements of $A_\gamma$ are called \emph{$\gamma$-homogeneous}, and we write $\deg(a) = \gamma$ for $a \in
A_{\gamma}\setminus \{0\}.$
When it is clear from context that a ring $A$ is graded by the group $\Gamma$, we simply say that $A$ is a \emph{graded ring}. If $A$ is an $R$-algebra, then $A$ is called a \emph{graded algebra} if it  is a graded ring and each $A_{\gamma}$ is an $R$-submodule.

An ideal $I \subseteq A$ is a \textit{graded ideal} if $I \subseteq \sum_{\gamma \in \Gamma} I \cap A_\gamma$. Graded left ideals, graded right ideals, graded subrings, and graded subalgebras are defined in a similar manner. If $H$ is a set of homogeneous elements in $A$, the ideal generated by $H$ is a graded ideal. Likewise, the left and right ideals generated by $H$ are graded. A \textit{graded homomorphism} of $\Gamma$-graded rings is a homomorphism $f: A \to B$ such that $f(A_\g) \subseteq B_\g$ for every $\g \in \Gamma$.
Finally, we say that a $\Gamma$-graded ring $A$ has \textit{homogeneous local units} (or \textit{graded local units}) if $A$ is locally unital, and the set of local units can be chosen to be a subset of $A_\ep$.

A topological groupoid $\G$ is called \textit{$\Gamma$-graded} if it can be partitioned by  clopen subsets $\G = \bigsqcup_{\gamma \in \Gamma} \G_\gamma$, such that $\G_\gamma \G_\delta \subseteq \G_{\gamma \delta}$ for  every $\gamma, \delta \in \Gamma$.
Equivalently $\G$ is $\Gamma$-graded if there is a continuous homomorphism $\kappa: \G \to \Gamma$. We can show the definitions are equivalent by setting $\G_\gamma = \kappa^{-1}(\{\gamma\})$. If $\kappa:\G \to \Gamma$ defines the grading on $\G$, we call it the \textit{degree map}.
We use the notation $\G_\gamma x = \G_\gamma \cap \G x$ and ${ x\G_\gamma} = {x\G} \cap \G_\gamma$ for $x \in \G^{(0)}$ and $\gamma \in \Gamma$.

We say a subset $X\subseteq \G$ is  $\gamma$-homogeneous if $X\subseteq \G_\gamma$.  Obviously, the unit space is $\varepsilon$-homogeneous and if  $X$ is $\gamma$-homogeneous then $X^{-1}$ is $\gamma^{-1}$-homogeneous. Moreover, ${\G_\gamma}^{-1} = \G_{\gamma^{-1}}$ for all $\gamma \in \Gamma$. For a $\Gamma$-graded ample groupoid, we write $B^{\rm co}_{\gamma}(\G)$ \index{$B^{\rm co}_{\gamma}(\G)$, $B_{*}^{\rm co}(\G)$} for the set of all $\gamma$-homogeneous compact open bisections of $\G$. For the set of all homogeneous compact open bisections, we use the notation:
\[
B_{*}^{\rm co}(\G)=\bigcup_{\gamma\in\Gamma} B^{\rm co}_{\gamma}(\G) \subseteq B^{\rm co}(\G).
\]
In Proposition \ref{inv-semigroup prop}, we proved that $B^{\rm co}(\G)$ is an inverse semigroup, and it is readily apparent that $B^{\rm co}_*(\G)$ is an inverse subsemigroup of $B^{\rm co}(\G)$.
In addition, $B_*^{\rm co}(\G)$ is a base of compact open bisections for $\G$. Indeed, since $B^{\rm co}(\G)$ is a base for $\G$, it suffices to show that every $B\in B^{\rm co}(\G)$ is a union of sets in $B^{\rm co}_*(\G)$. This is almost trivial, for if $B \in B^{\rm co}(\G)$ then $B = \bigcup_{\gamma \in \Gamma} B \cap \G_\gamma$ and $B \cap \G_\gamma \in B_\gamma^{\rm co}(\G)$. The next two results are from \cite[Lemma 3.1]{clark2015equivalent}.

\begin{proposition} \label{grading}
	If $\G = \bigsqcup_{\gamma \in \Gamma}\G_\gamma$ is a $\Gamma$-graded ample groupoid, then $A_R(\G) = \bigoplus_{\gamma \in \Gamma}A_R(\G)_\gamma$ is a $\Gamma$-graded algebra with homogeneous local units, where: 
	\begin{align*}
	&A_R(\G)_\gamma = \left\{f \in A_R(\G)\mid \supp f  \subseteq \G_\gamma\right\} && \text{for all } \gamma \in \Gamma.
	\end{align*}
\end{proposition}

\begin{proof}
	From Proposition \ref{spann}, it follows that
	\[
	A_R(\G) = \Span_R\{\bm{1}_B \mid B \in B^{\rm co}_*(\G) \} = \sum_{\gamma \in \Gamma} \Span_R\{\bm{1}_B \mid B \in B^{\rm co}_\gamma(\G)\} = \sum_{\gamma \in \Gamma}A_R(\G)_\gamma.
	\]
	It is clear that $A_R(\G)_\gamma \cap \big(\sum_{\delta \ne \gamma} A_R(\G)_\delta\big) = \{0\}$ for all $\gamma \in \Gamma$, so we have $A_R(\G) = \bigoplus_{\gamma \in \Gamma} A_R(\G)_\gamma$. Now for all $f \in A_R(\G)_\gamma$ and $g \in A_R(\G)_\delta$, we have $\supp (f*g) \subseteq \supp (f) \supp (g)  \subseteq \G_\gamma\G_\delta \subseteq \G_{\gamma \delta}$, and thus $f*g \in A_R(\G)_{\gamma\delta}$. Therefore $A_R(\G)_\gamma * A_R(\G)_\delta \subseteq A_R(\G)_{\gamma \delta}$. It follows from Proposition \ref{units}, and the fact that $\G^{(0)} \subseteq \G_\varepsilon$, that $A_R(\G)$ has homogeneous local units.
\end{proof}

\begin{lemma} \label{mut disj}
	If $\G$ is a $\Gamma$-graded Hausdorff ample groupoid, every $f \in A_R(\G)$ can be expressed as a finite sum $f = \sum_{i=1}^n r_i \bm{1}_{B_i}$, where $r_1, \dots, r_n \in R$, and $B_1, \dots, B_n \in B_*^{\rm co}(\G)$ are mutually disjoint.
\end{lemma}

\begin{proof}
	Since $\G$ is Hausdorff, every homogeneous compact open bisection is closed, so $B_*^{\rm co}(\G)$ is closed under finite intersections and relative complements. The statement now follows from Proposition \ref{ARG char}.
\end{proof}

\begin{example}
	Recall, from Example \ref{groupoid ex}~(\ref{transformation groupoid}), the definition of the transformation groupoid $G \times X$, associated to a group $G$ and a $G$-set $X$. Now assume that $X$ is a locally compact, totally disconnected, Hausdorff topological space, and for each $g \in G$ the map $\rho_g: X \to X$, $\rho_g(x) = g \cdot x$, is continuous. If we assign the discrete topology to $G$ and the product topology to $G \times X$, then $G \times X$ is an ample groupoid. It is easy to verify that this is a $G$-graded groupoid with homogeneous components
	$(G \times X)_g = \{g\} \times X$ for all $g \in G$.
	The Steinberg algebra of $G \times X$ turns out (see \cite{beuter2017interplay}) to be the \textbf{skew group ring} $C_R(X) \star G$, associated to a certain action of $G$ on $C_R(X)$, canonically induced by the action of $G$ on $X$.
\end{example}

One can generalise this example quite profitably, by replacing the group action with something more general called a \textit{partial group action} (see \cite[Definition 2.1]{exel2017partial}). In doing so, one obtains a class of algebras so general that it includes all Leavitt path algebras (see \cite[Theorem 3.3]{gonccalves2014leavitt}) and other interesting things, like the \textit{partial group algebras} that were studied in \cite{dokuchaev2000partial} and \cite{hazrat2017graded}.

\section{The path space and boundary path groupoid of a graph} \label{CHAPTER: GRAPHS AND GROUPOIDS}

Part 2 is structured as follows. In \S\ref{graph concepts}, we define directed graphs and introduce some terminology. In \S\ref{path space}, we introduce a topological space called the \textit{path space} of a graph. The path space of a graph is the set of all finite and infinite paths, with a topology described explicitly by a base of open sets. Generalising \cite[Theorem 2.1]{webster2014path}, we prove in Theorem \ref{path space topology} that for graphs of any cardinality, the path space is locally compact and Hausdorff. We also determine which graphs have a second-countable, first-countable, or $\sigma$-compact path space. In \S\ref{graph grpd}, we use the path space (or more precisely, a closed subspace called the boundary path space) to define the \textit{boundary path groupoid} associated to a graph. We prove it is ample and study its local structure from a topological and an algebraic point of view.

\begin{remark}
Perhaps as an artefact of its history, many fundamental properties of the boundary path groupoid were absorbed into folklore. Some proofs were never written, and others were written at a higher level of generality, and not all in one place, making them difficult to relate back to our present needs.
For instance, we could not find a complete proof that the boundary path groupoid is an ample groupoid, even though this fact was used in all the early papers  that pioneered the use of groupoid methods for Leavitt path algebras \cite{orloff2016using,clark2016using,clark2015equivalent}. The groupoid approach to Leavitt path algebras is particularly well-suited, compared to traditional, purely algebraic methods, for dealing with graphs of large cardinalities. Therefore, it is important to make sure that the theorems used to justify these methods can be proved without assuming graphs are countable. This is something that we achieve here, in Theorem \ref{path space topology} and Theorem \ref{weeks}.
\end{remark}

\subsection{Graphs} \label{graph concepts}

In this section, we introduce the necessary terminology and conventions pertaining to graphs. We always use the word graph to mean a \textit{directed} graph, defined as follows.

\begin{definition}
	A \textbf{graph} is a system $E = (E^0, E^1, r, s)$, \index{$E$, $E^0$, $E^1$, $r$, $s$} where $E^0$ is a set whose elements are called \textit{{vertices}}, $E^1$ is a set whose elements are called \textit{{edges}}, $r: E^1 \to E^0$ is a map that associates a \textit{{range}} to every edge, and $s: E^1 \to E^0$ is a map that associates a \textit{{source}} to every edge.
\end{definition}

A \textit{countable graph} is one where $E^0$ and $E^1$ are countable sets.
A \textit{row-finite} (resp., \textit{row-countable}) graph is one in which $s^{-1}(v)$ is finite (resp., countable) for every $v \in E^0$. 
If $e$ is an edge with $s(e) = v$ and $r(e) = w$ then we say that $v$ \textit{{emits}} $e$ and $w$ \textit{{receives}} $e$. A \textit{{sink}} is a vertex that emits no edges
and an \textit{infinite emitter} is a vertex that emits infinitely many edges. 	If $v \in E^0$ is either a sink or an infinite emitter (that is, $s^{-1}(v)$ is either empty or infinite) then $v$ is called \textit{singular}, and if $v$ is not singular then it is called \textit{regular}. A vertex that neither receives nor emits any edges is called an \textit{{isolated vertex}}.

A \textit{{finite path}} is a finite sequence of edges $\alpha = \alpha_1 \alpha_2 \dots \alpha_n$ such that $r(\alpha_i) = s(\alpha_{i+1})$ for all $i = 1, \dots, n-1$. The \textit{length} of the path $\alpha$ is $|\alpha|= n$. Reusing notation and terminology, we shall say that $s(\alpha) = s(\alpha_1)$ is the \textit{{source}} of the path, and $r(\alpha) = r(\alpha_n)$ is the \textit{{range}} of the path.
By convention, vertices $v \in E^0$ are regarded as finite paths of zero length, with $r(v) = s(v) = v$.	 If $v, w \in E^0$, we write $v \ge w$ \index{$v\ge w$} if there exists a finite path $\alpha$ with $s(\alpha) = v$ and $r(\alpha) = w$.   If a finite path $\alpha$ of positive length satisfies $r(\alpha) = s(\alpha) = v$, then $\alpha$ is called a \textit{{closed path}} based at $v$.
A closed path $\alpha$ with the property that none of the vertices $s(\alpha_1), \dots, s(\alpha_{|\alpha|})$ are repeated is called a \textit{{cycle}}, and a graph that has no cycles is called \textit{{acyclic}}. An \textit{exit} for a finite path $\alpha$ is an edge $f \in E^1$ with $s(f) = s(\alpha_i)$ for some $1 \le i \le |\alpha|$, but $f \ne \alpha_i$. 

An \textit{infinite path} is, predictably, an infinite sequence of edges $p = p_1 p_2 p_3 \dots$ such that $r(p_i) = s(p_{i+1})$ for $i = 1,2, \dots$. Again, $s(p) = s(p_1)$ is called the source of the infinite path $p$. We let $|p| = \infty$ if $p$ is an infinite path.
%
We use the notation $E^\star$ \index{$E^\star$, $E^\infty$} for the set of finite paths (including vertices), and $E^\infty$ for the set of infinite paths.

Paths can be concatenated if their range and source agree. If $\alpha, \beta \in E^\star$ have positive length and $r(\alpha) = s(\beta)$, then $\alpha\beta= \alpha_1\dots \alpha_{|\alpha|}\beta_1 \dots \beta_{|\beta|} \in E^\star$. If $p \in E^\infty$ has $r(\alpha) = s(p)$, then $\alpha p = \alpha_1\dots \alpha_{|\alpha|}p_1 p_2\ldots \in E^\infty$. If $v \in E^0$ and $x \in E^\star \cup E^\infty$ has $s(x) = v$, then $vx = x$ by convention. Likewise, if $\alpha \in E^\star$ has $r(\alpha) = v$ then $\alpha v = \alpha$. If $\alpha \in E^\star$, $x \in E^\star \cup E^\infty$, and  $x = \alpha x'$ for some $x' \in E^\star \cup E^\infty$, then we say that $\alpha$ is an \textit{initial subpath} of $x$. In particular, $s(\alpha)$ is considered an initial subpath of $\alpha$.

Let $E^0_{\rm sing} = \{v \in E^0 \mid v \text{ is singular}\}$  and $E^0_{\rm reg} = \{v \in E^0 \mid v \text{ is regular}\}$. Using the terminology of \cite{ webster2014path}, we define the set of \textit{boundary paths} as
\[
\partial E = E^\infty \cup \left\{\alpha \in E^\star\mid r(\alpha) \in E^0_{\rm sing}\right\}.
\]
We employ the following notation from now on: \index{$vE^1$, $vE^\star$, $vE^\infty$, $v\partial E$} \index{$E^\star \times_r E^\star$}
\begin{align*} 
&vE^1 = \{e \in E^1 \mid s(e) = v\},  &&vE^\star = \{\alpha \in E^\star \mid s(\alpha) = v\},\\ \notag
&vE^\infty = \{p \in E^\infty \mid s(p) = v\}, &&v\partial E = \{x \in \partial E \mid s(x) = v\}, \\ \notag
&E^\star \times_r E^\star = \big\{(\alpha, \beta) \in E^\star \times E^\star \mid r(\alpha) = r(\beta) \big\}.
\end{align*}

\subsection{The path space of a graph} \label{path space}

Throughout this section, assume $E = (E^0, E^1, r, s)$ is an arbitrary graph. The \textit{path space} of $E$ is $E^\star \cup E^\infty$, the set of all finite and infinite paths, and the \textit{boundary path space} is $\partial E$, the set of paths that are either infinite or end at a singular vertex. We now set out to define a suitable topology on the path space.
For a finite path $\alpha \in E^\star$, we define the \textit{cylinder} set \index{$C(\alpha)$, $C(\alpha, F)$}
\begin{align} \label{cylinder}
C(\alpha)= \big\{\alpha x \mid x \in E^\star \cup E^\infty, r(\alpha)=s(x)\big\} \subseteq E^\star \cup E^\infty.
\end{align}
It is easy to see that the intersection of two cylinders is either empty or a cylinder. Indeed, if $x \in C(\alpha) \cap 
C(\beta)$ then $x = \alpha y = \beta z$ for some $y,z \in E^\star \cup E^\infty$. If $|\alpha| \le |\beta|$ then $\alpha$ is an initial subpath of $\beta$, implying $C(\beta) \subseteq C(\alpha)$. In symbols:
\[
C(\alpha) \cap C(\beta)
= \begin{cases}
C(\beta) & \text{if }\alpha\text{ is an initial subpath of }\beta \\
C(\alpha) & \text{if }\beta\text{ is an initial subpath of }\alpha \\
\emptyset & \text{otherwise.}
\end{cases}
\]
This is all we need to conclude that the collection of cylinder sets is a base for a topology on $E^\star \cup E^\infty$. As the authors of \cite{kumjian1997graphs} have stated, the subspace $E^\infty \subseteq E^\star \cup E^\infty$ with the cylinder set topology is homeomorphic (in the canonical way) to a subspace of $\prod_{n = 1}^\infty E^1$, where $E^1$ is discrete and the product has the product topology.
In particular, the cylinder sets generate a Hausdorff topology on $E^\infty$, and if $E$ is row-finite, that topology is locally compact. However, the cylinder set topology generated by the sets (\ref{cylinder}) is not Hausdorff (or even $T_1$) on the whole set $E^\star \cup E^\infty$, because a finite path cannot be separated from a proper initial subpath. In order to have enough open sets in hand for a Hausdorff topology, we define a base of open sets called \textit{generalised cylinder} sets:
\begin{align}\label{C-basis}
&C(\alpha, F)= C(\alpha) \setminus \bigcup_{e \in F} C(\alpha e) ;& &\alpha \in E^\star,\ F\subseteq r(\alpha)E^1 \text{ is finite.}
\end{align}
We shall write $F \subseteq_{\rm finite} vE^1$ \index{$\subseteq_{\rm finite}$} to mean that $F$ is a finite subset of $vE^1$.
The next lemma (a generalisation of \cite[Lemma 2.1]{kumjian1997graphs}) shows that the collection of generalised cylinders is closed under intersections, so it is a base for a topology on $E^\star \cup E^\infty$. With the generalised cylinder set topology on $E^\star \cup E^\infty$, every finite path is an isolated point unless its range is an infinite emitter. 

\begin{lemma} \label{int lemma 0}
	If $\alpha, \beta \in E^\star$, $|\alpha| \le |\beta|$, $F \subseteq_{\rm finite} r(\alpha)E^1$, and $H\subseteq_{\rm finite} r(\beta)E^1$, then
	\[
	C(\alpha , F) \cap C(\beta , H)
	= \begin{cases}
	C(\beta , F \cup H) & \text{if } \beta = \alpha \\
	C(\beta , H) &\text{if } \exists\ \delta \in E^\star,\ |\delta| \ge 1, \ \beta = \alpha\delta, \text{ and } \delta_1 \notin F \\
	\emptyset & \text{otherwise.}
	\end{cases}
	\]
\end{lemma}

\begin{proof}
	By definition of $C(\alpha,F)$ and $C(\beta, H)$, we have
	\begin{equation} \label{int eq}
	C(\alpha , F) \cap C(\beta , H) = C(\alpha) \cap C(\beta) \setminus
	\left(
	\bigcup_{e \in F}C(\alpha e) \cup \bigcup_{e \in H}C(\beta e)
	\right).
	\end{equation}
	If $\beta = \alpha$, the right hand side of (\ref{int eq}) is $C(\beta , F \cup H)$. If $\beta = \alpha \delta$ ($|\delta| \ge 1$) and $\delta_1 \notin F$ then $C(\beta) \cap C(\alpha) = C(\beta)$ does not meet $\bigcup_{e \in F}C(\alpha e)$, so the right hand side of (\ref{int eq}) is $C(\beta , H)$. If $\beta = \alpha \delta$ and $\delta_1 \in F$, then $C(\beta) \cap C(\alpha) = C(\beta) = C(\alpha \delta_1 \dots \delta_{|\delta|}) \subseteq C(\alpha \delta_1) \subseteq \bigcup_{e \in F}C(\alpha e)$, so the right hand side of (\ref{int eq}) is empty. If $\alpha$ is not an initial subpath of $\beta$ then $C(\alpha) \cap C(\beta) = \emptyset$.
\end{proof}

To apply Steinberg's theory from Part \ref{CHAPTER: STEINBERG ALGEBRAS}, it is critical that the induced topology on the boundary path space $\partial E \subseteq E^\star \cup E^\infty$ is locally compact and Hausdorff. We proceed by proving that the topology on the path space $E^\star \cup E^\infty$, generated by the base in (\ref{C-basis}), is locally compact and Hausdorff, and that $\partial E$ is closed in $E^\star \cup E^\infty$. As it were, this base is well-chosen: the basic open sets themselves are compact in the Hausdorff topology that they generate.

The proof of the theorem below is essentially the same as \cite[Theorem 2.1]{webster2014path}, just written slightly differently so that it does not use any assumptions of countability. The main idea is to equip $\mathbb{P}(E^\star)$, i.e., the power set of $E^\star$, with a compact Hausdorff topology, and show that $E^\star \cup E^\infty$ is homeomorphic to a locally compact subspace $\mathbb{S} \subset \mathbb{P}(E^\star)$.
%
\begin{theorem} \label{path space topology}
	The collection  (\ref{C-basis}) of generalised cylinder sets is a base of compact open sets for a locally compact Hausdorff topology on $E^\star \cup E^\infty$.
\end{theorem}

\begin{proof}
Let $\{0,1\}$ have the discrete topology. The product space $\{0,1\}^{E^\star}$ is compact by Tychonoff's Theorem, and Hausdorff because products preserve the Hausdorff property. There is a canonical bijection from $\mathbb{P}(E^\star)$ to $\{0,1\}^{E^\star}$, which transfers a compact Hausdorff topology to $\mathbb{P}(E^\star)$. For the first part of the proof, we work entirely in the space $\mathbb{P}(E^\star)$. The topology on $\mathbb{P}(E^\star)$, by definition, is generated by the base of open sets
\[
\big\{[P,N]  \mid P, N \subseteq_{\rm finite} E^\star\big\},
\]
where we define
\[
[P,N] = \big\{A \in \mathbb{P}(E^\star) \mid P \subseteq A,\ N \subseteq E^\star \setminus A\big\}.
\]
Note that $[P,N] = \emptyset$ if $P \cap N \ne \emptyset$.

	Define the subspace $\mathbb{S} \subset \mathbb{P}(E^\star)$ to be the set of subsets $A \subseteq E^\star$ such that:
	\begin{itemize}
	\item $A \ne \emptyset$ and for all $\alpha \in A$, every initial subpath of $\alpha$ is in $A$;
	\item For every $0 \le n < \infty$, there is at most one path of length $n$ in $X$.
	\end{itemize}
	We claim that $\mathbb{S} \cup \{\emptyset\}$ is closed in $\mathbb{P}(E^\star)$. Suppose $A \in \mathbb{P}(E^\star) \setminus \big(\mathbb{S} \cup \{\emptyset\}\big)$. If $A$ contains two distinct paths $\alpha$ and $\beta$ of the same length, then $\big[\{\alpha, \beta\}, \emptyset\big]$ is open, contains $A$, and does not meet $\mathbb{S} \cup \{\emptyset\}$. If there is some $\alpha \in A$ and an initial subpath $\beta$ of $\alpha$ such that $\beta \notin A$, then $\big[\{\alpha\}, \{\beta\}\big]$ is open, contains $A$, and does not meet $\mathbb{S} \cup \{\emptyset\}$. Failing this, $A \in \mathbb{S} \cup \{\emptyset\}$, which we assumed is false. Therefore $\mathbb{S} \cup \{\emptyset \}$ is closed in $\mathbb{P}(E^\star)$, which implies it is compact.
	
	We now work out what the subspace topology is on $\mathbb{S}$. Let $P,N \subseteq_{\rm finite} E^\star$. If $[P,N] \cap \mathbb{S} \ne \emptyset$ then $P$ contains a unique path $\rho$ of maximal length (because of the way $\mathbb{S}$ is defined) and $[P,N] \cap \mathbb{S} = \big[\{\rho\}, N'\big] \cap \mathbb{S}$ where 
	\[N' = \{\eta \in N \mid \rho \text{ is an initial subpath of } \eta\}.\]
	Therefore, the topology on $\mathbb{S}$ is generated by basic open sets of the form $\big[\{\rho\}, N'\big]\cap \mathbb{S}$ where $\rho \in E^\star$ and $N' \subseteq E^\star$ is a finite set of paths that are proper extensions of $\rho$.
	
	Note that $\mathbb{S} = \bigsqcup_{v \in E^0} \big[\{v\},\emptyset\big]\cap \mathbb{S}$. For each $v \in E^0$, the set $\big[\{v\}, \emptyset\big]$ is closed in $\mathbb{P}(E^\star)$ because $\mathbb{P}(E^\star) \setminus \big[\{v\}, \emptyset\big] = \big[\emptyset, \{v\}\big]$ is open.
	Since $\big[\{v\}, \emptyset\big] \cap \mathbb{S} = \big[\{v \}, \emptyset\big] \cap \big(\mathbb{S} \cup \{\emptyset\}\big)$ and $\mathbb{S} \cup \{\emptyset\}$ is closed in $\mathbb{P}(E^\star)$, we have that $\big[\{v\}, \emptyset\big] \cap \mathbb{S}$ is closed in $\mathbb{P}(E^\star)$, and therefore compact. This proves that $\mathbb{S}$ is locally compact, because it is Hausdorff and every point has a compact neighbourhood.
%
	
	Now we show that $E^\star \cup E^\infty$ is homeomorphic to $\mathbb{S}$. Define the map $\Psi: E^\star \cup E^\infty \to \mathbb{S}$,
	\begin{align*}	\Psi(x) = \{\nu \in E^\star \mid \nu \text{ is an initial subpath of } x\}.\end{align*}
	It is clear that $\Psi$ is a bijection.	Let $\rho \in E^\star$ and let $N' \subseteq E^\star$ be a finite set of paths that properly extend $\rho$. Then
	\begin{align*}
	\Psi^{-1}\big([\{\rho\}, N'] \cap \mathbb{S}\big) = C(\rho) \setminus \bigcup_{\rho \beta \in N'} C(\rho \beta) = \bigcap_{\rho \beta \in N'} C(\rho) \setminus C(\rho\beta).
	\end{align*}
	It is not difficult to see that for each $\rho \beta \in N'$, the set
	\[
	C(\rho ) \setminus C(\rho\beta) = C\big(\rho, \{\beta_1\}\big) \cup C\big(\rho \beta_1, \{\beta_2\}\big) \cup \cdots \cup C\big(\rho \beta_{|\beta|-1}, \{\beta_{|\beta|}\}\big)
	\]
	is open. Therefore $\Psi^{-1}\big([\{\rho\}, N']\cap \mathbb{S}\big)$ is open in $E^\star \cup E^\infty$. Consequently, $\Psi$ is continuous. If $\alpha \in E^\star$ and $F \subseteq_{\rm finite} r(\alpha)E^1$, then $C(\alpha, F)$ is mapped to an open set in $\mathbb{S}$:
	\[
	\Psi\big(C(\alpha, F)\big) = \big[\{\alpha\}, N'\big]\cap \mathbb{S}
	\]
	where $N' = \{\alpha e \mid e \in F \}$. It follows that $\Psi$ is a homeomorphism and $E^\star \cup E^\infty$ is Hausdorff.

	Since we showed that $[\{v\}, \emptyset] \cap \mathbb{S}$ is compact, it follows that $C(v) = \Psi^{-1}([\{v\}, \emptyset] \cap \mathbb{S})$ is compact, for all $v \in E^0$. To show that $C(\alpha)$ is compact for all $\alpha \in E^\star$, we proceed by induction on the length of $\alpha$. If $e \in E^1$, then $C(s(e)) \setminus C(e) = C(s(e), \{e\})$
	is a basic open set, so $C(e)$ is closed in $C(s(e))$, hence compact.
	Assume $C(\alpha)$ is compact for any $\alpha \in E^\star$ with $|\alpha| = n$. If $\mu \in E^\star$ has $|\mu| = n+1$ then let $\mu' = \mu_1 \mu_2 \dots \mu_n$. We have that 
	$C(\mu') \setminus C(\mu) = C(\mu' , \{\mu_{n+1}\})$
	is a basic open set, so $C(\mu)$ is closed in $C(\mu')$, hence compact. By induction, $C(\alpha)$ is compact for arbitrary $\alpha \in E^\star$. Finally, if $F \subseteq_{\rm finite} r(\alpha)E^1$ then $C(\alpha) \setminus~C(\alpha,F) = \bigcup_{e \in F}C(\alpha e)$ is open, so $C(\alpha , F)$ is compact.
\end{proof}

Recall that a topological space is called \textit{second-countable} if it has a countable base, \textit{first-countable} if every point has a countable neighbourhood base, and \textit{$\sigma$-compact} if it is a countable union of compact subsets.

\begin{theorem} \label{counting}
The path space $E^\star \cup E^\infty$ is:
	\begin{enumerate}[\rm (1)]
		\item \label{count 1} second-countable if and only if $E$ is a countable graph;
		\item \label{count 2} first-countable if and only if $E$ is a row-countable graph;
		\item \label{count 3} $\sigma$-compact if and only if $E^0$ is countable.
	\end{enumerate}
\end{theorem}

\begin{proof}
	(\ref{count 1}) If $E$ is a countable graph (i.e., $E^0 \cup E^1$ is countable) then $E^\star$ is countable. The base of open sets (\ref{C-basis}) is countable too, because there are only countably many pairs $(\alpha, F)$ where $\alpha \in E^\star$ and $F\subseteq_{\rm finite}r(\alpha)E^1$. This proves the topology is second-countable. Conversely, if one of $E^0$ or $E^1$ is uncountable, then one of $\{C(v) \mid v \in E^0\}$ or $\{C(e) \mid e \in E^1 \}$ is an uncountable set of pairwise disjoint open sets, so $E^\star \cup E^\infty$ is not second-countable.
	
	(\ref{count 2}) Notice that the following sets are neighbourhood bases at $\alpha \in E^\star$ and $p \in E^\infty$ respectively:
	\begin{align*}
	&\mathcal{N}_\alpha = \big\{C(\alpha , F)\mid F \subseteq_{\rm finite} r(\alpha)E^1 \big\}, && \mathcal{N}_p =	\big\{C(p_1\ldots p_m)\mid m \ge 1\big\}.
	\end{align*}
	Regardless of the graph, $\mathcal{N}_p$ is countable for every $p \in E^\infty$. If a finite path $\alpha \in E^\star$ has the property that $r(\alpha)E^1$ is countable, then $\mathcal{N}_\alpha$ is countable, because there are only countably many finite subsets $F$ of $r(\alpha)E^1$. So, for every row-countable graph $E$, the path space $E^\star \cup E^\infty$ is first-countable.  Conversely, suppose there exists $v \in E^0$ such that $vE^1$ is uncountable. Towards a contradiction, assume $v$ has a countable neighbourhood base $\mathcal{B}_v = \{B_1, B_2, \dots,\}$. By replacing $B_n$, for all $n \ge 1$, with a set of the form $C(v,F_n)\subseteq B_n$, where $F_n \subseteq_{\rm finite} vE^1$, we have a countable neighbourhood base for $v$ of the form $\mathcal{C}_v = \{C(v, F_1), C(v,F_2), \dots\}$. Since $\bigcup_{n = 1}^\infty F_n$ is countable, one can choose $e \in vE^1 \setminus \bigcup_{n = 1}^\infty F_n$. Then every neighbourhood of $v$ contains $e$, which is absurd, because the space is Hausdorff. Therefore $E^\star \cup E^\infty$ is first-countable if and only if $E$ is row-countable.
	
	(\ref{count 3}) If $E^0$ is countable then the path space is $\sigma$-compact, because $E^\star \cup E^\infty = \bigcup_{v \in E^0}C(v)$ and $C(v)$ is compact for every $v \in E^0$, by Theorem \ref{path space topology}.  For the converse, suppose $E^\star \cup E^\infty$ is $\sigma$-compact. Then there is a sequence of compact subsets $(K_n)_1^\infty$ such that $E^\star \cup E^\infty = \bigcup_{n = 1}^\infty K_n$. Each $K_n$ is compact, so it can be covered by a finite subcover of $\{C(v) \mid v \in E^0 \}$, implying that there is a countable set $S \subseteq E^0$ such that $E^\star \cup E^\infty = \bigcup_{v \in S} C(v)$. But this implies $S = E^0$ because $C(v)$ and $C(w)$ are disjoint unless $v = w$.
\end{proof}

We now prove an easy fact that forms a bridge to the next section, where we shall construct a groupoid with unit space $\partial E = E^\infty \cup \big\{\alpha \in E^\star \mid r(\alpha) \in E^0_{\rm sing}\big\}$.

\begin{proposition} \label{lem-closed}
	The boundary path space $\partial E$ is closed in $E^\star \cup E^\infty$.
\end{proposition}

\begin{proof}
	The complement of $\partial E$ consists of isolated points. Indeed, if $\mu \in (E^\star \cup E^\infty) \setminus \partial E$, then $r(\mu)$ is a regular vertex, and $C(\mu , r(\mu)E^1) = \{\mu\}$ is open in $E^\star \cup E^\infty$.
\end{proof}

An immediate consequence of Theorem \ref{path space topology} and Proposition \ref{lem-closed} is that $\partial E$ is a locally compact Hausdorff space with the base of compact open sets: \index{$Z(\alpha)$, $Z(\alpha, F)$}
\begin{align*} \label{Z-base boundary}
&Z(\alpha, F) = C(\alpha,F) \cap \partial E; & &\alpha \in E^\star,\ F \subseteq_{\rm finite} r(\alpha)E^1.
\end{align*}
For $\alpha \in E^\star$, we define $Z(\alpha) = Z(\alpha, \emptyset)$, which is the same as $Z(\alpha) = C(\alpha) \cap \partial E$. As it were, the sets $Z(\alpha, F)$ are very rarely empty. In particular, $Z(\alpha)\ne \emptyset$ for all $\alpha \in E^\star$; in other words, every finite path can be extended to a boundary path.

\begin{lemma} \label{non-empty}
	Let $\alpha \in E^\star$ and let $F \subseteq_{\rm finite} r(\alpha)E^1$. Then $Z(\alpha, F) = \emptyset$ if and only if $r(\alpha)$ is a regular vertex and $F =  r(\alpha)E^1$.
\end{lemma}
\begin{proof}
	$(\Rightarrow)$ Assume $Z(\alpha, F) = \emptyset$. If $r(\alpha)$ were a singular vertex then it would imply $\alpha \in Z(\alpha,F)$. Therefore $r(\alpha)$ is regular, so $r(\alpha)E^1 \ne \emptyset$. Towards a contradiction, assume $F$ is a proper subset of $r(\alpha)E^1$. Then there exists some $x_1 \in r(\alpha)E^1 \setminus F$. Assume that we have a path $x_1 x_2 \dots x_n \in r(\alpha)E^\star$. If $r(x_n)$ is a sink, let $x = x_1 \dots x_n$. Otherwise, let $x_{n+1} \in r(x_n)E^1$. Inductively, this constructs $x \in r(\alpha)\partial E$ such that $\alpha x \in Z(\alpha, F)$. Since this is a contradiction, it proves $F = r(\alpha)E^1$.
	
	$(\Leftarrow)$ If $r(\alpha)$ is regular, then $Z(\alpha) = \bigcup_{e \in r(\alpha) E^1}Z(\alpha e)$, so $Z(\alpha, r(\alpha)E^1) = \emptyset$.
\end{proof}

\begin{theorem} \label{counting counting}
	The boundary path space $\partial E$ is:
	\begin{enumerate}[\rm (1)]
		\item second-countable if and only if $E$ is a countable graph,
		\item first-countable if and only if $E$ is a row-countable graph, and
		\item $\sigma$-compact if and only if $E^0$ is countable.
	\end{enumerate}
\end{theorem}

\begin{proof}
	Together with Lemma \ref{non-empty}, the proof is almost identical to the relevant parts of Theorem \ref{counting}.
\end{proof}

\subsection{The boundary path groupoid} \label{graph grpd}

In this section, we define the boundary path groupoid of a graph (see \cite[Example 2.1]{clark2015equivalent}) and investigate some of its algebraic and topological properties. Throughout, let $E = (E^0, E^1, r, s)$ be an arbitrary graph.

Define the \textit{one-sided shift map} $\sigma: \partial E \setminus E^0 \to \partial E$ as follows: \index{$\sigma$}
\[
\sigma(x) = 
\begin{cases}
r(x) &\text{ if } x \in E^\star \cap \partial E \text{ and } |x|=1 \\
x_2 \dots x_{|x|} &\text{ if } x \in E^\star \cap \partial E \text{ and } |x| \ge 2\\
x_2 x_3 \dots &\text{ if } x \in E^\infty
\end{cases}
\]
The $n$-fold composition $\sigma^n$ is defined on paths of length $\ge n$ and we understand that $\sigma^0: \partial E \to \partial E$ is the identity map.	

\begin{definition}
	Let $k$ be an integer and let $x, y\in \partial E$. We say that $x$ and $y$ are \textbf{tail equivalent with lag $k$}, written $x \sim_k y$, \index{$\sim_k$} if there exists some $n \ge \max\{0,k\}$ such that 
	\[
	\sigma^n(x) = \sigma^{n - k}(y).
	\]
	If an integer $k$ exists such that $x \sim_k y$, we say that $x$ and $y$ are \textit{tail equivalent}, and write $x \sim y$. 
\end{definition}

An equivalent definition is that $x \sim_k y$ if there exists $(\alpha, \beta) \in E^\star\times_r E^\star$ and $z \in r(\alpha)\partial E$, such that $x = \alpha z$, $y = \beta z$, and $|\alpha| - |\beta| = k$. Something that is potentially counter-intuitive about these relations is that the lag is not necessarily unique: it is possible to have $x \sim_k y$ and $x \sim_\ell y$ even when $k \ne \ell$. 
It is straightforward to prove from the definition that for all $x, y, z \in \partial E$:
\begin{gather*}
x \sim_0 x ,\\
x \sim_k y \implies y \sim_{-k} x, \\
x \sim_k y \text{ and } y \sim_\ell z \implies x \sim_{k + \ell}z,\\
x \sim_k y \implies x, y \in  E^\star \text{ or } x, y \in E^\infty.
\end{gather*}
This shows that $\sim$ is an equivalence relation on $\partial E$ that respects the partition between finite and infinite paths.
\begin{definition}
	The \textbf{boundary path groupoid} of a graph $E$ is \index{$\G_E$}
	\begin{align*}
	\G_E &= \big\{(x, k, y)\mid x, y \in \partial E ,\ x \sim_k y\big\}\\\
	&= \big\{(\alpha x, |\alpha| - |\beta|, \beta x)\mid (\alpha, \beta) \in E^\star\times_r E^\star, x \in r(\alpha)\partial E\big\}
	\end{align*}
	where a morphism $(x, k, y) \in \G_E$ has {domain} $y$ and {codomain} $x$. The composition of morphisms and their inverses are defined by the formulae:
	\begin{align*}
	&(x, k , y)(y, l,z)
	= (x, k + l, z), && (x, k, y)^{-1}=(y, -k, x).
	\end{align*}
\end{definition}
The unit space is $\G_E^{(0)} = \{(x, 0,x) \mid x \in \partial E \}$, which we silently identify with $\partial E$ (see Remark \ref{identify}). The orbits in $\partial E$ are tail equivalence classes.


\begin{example} \label{rosetwo}
	Consider this graph, called the \textit{rose with two petals}:
	\[
	R_2 \quad = \quad \xymatrix{
		\bullet_v \ar@(ul,dl)_e \ar@(ur, dr)^f
	}
	\]
	A standard diagonal argument proves that $\partial R_2$ is an uncountable set. There are uncountably many orbits in $\partial R_2$, but the topology on $\partial R_2$ is second-countable and even metrisable. In fact, it can be shown that $\partial R_2$ is homeomorphic to the Cantor set $\{0,1\}^\mathbb{N}$. 
\end{example}

A boundary path $p \in \partial E$ is called \textit{eventually periodic} if it is of the form $p = \mu \epsilon \epsilon \ldots \in E^\infty$ where $\mu, \epsilon \in E^\star$ and $\epsilon$ is a closed path of positive length. The following result is \cite[Proposition 4.2]{steinberg2018chain} except there appears to be a clash between our definitions of cycles and closed paths. We also prove it slightly more formally.

\begin{proposition}
	\label{GE isotropy}
	If $E$ is a graph and $p \in \partial E$, then the isotropy group at $p$ is:
	\begin{enumerate}[\rm (1)]
		\item\label{GE isotropy1}
		infinite cyclic if $p$ is eventually periodic;
		\item \label{GE isotropy2}
		trivial if $p$ is not eventually periodic.
	\end{enumerate}
\end{proposition}

\begin{proof}
	(\ref{GE isotropy1}) Assume $p = \mu \epsilon \epsilon \ldots \in E^\infty$ where $\mu, \epsilon \in E^\star$, $r(\mu) = s(\epsilon) = r(\epsilon)$, and assume $\epsilon$ is minimal in the sense that it has no initial subpath $\delta$ such that $\epsilon  = \delta^n$ for some $n > 1$. Let $(p,k,p) \in {p{(\G_E)}p}$ and suppose $k \ge 0$. Then $p \sim_k p$ implies that for all sufficiently large $n \ge 0$, we have $\sigma^{|\mu|+ n|\epsilon|+k}(p) = \sigma^{|\mu|+n|\epsilon|}(p)$. This yields:
	\[
	\sigma^{|\mu|+n|\epsilon|+k}(p) =   \sigma^k(\epsilon \epsilon \dots)  =  \sigma^{|\mu|+n|\epsilon|}(p) = \epsilon \epsilon \dots.
	\]
	Let $m = k \mod |\epsilon|$. Then $0 \le m < |\epsilon|$ and
	\[
	\sigma^k(\epsilon\epsilon \dots) = \sigma^m(\epsilon \epsilon \dots) = \epsilon_{m+1} \dots \epsilon_{|\epsilon|}\epsilon \epsilon \ldots = \epsilon_1 \dots \epsilon_m \epsilon \epsilon \dots.
	\]
	Since $\epsilon$ is minimal, this implies $m = 0$, so $k \mid |\epsilon|$. On the other hand, if $k < 0$ then $(p,-k,p) = (p,k,p)^{-1} \in {p{(\G_E)}p}$ and the same argument establishes $k \mid |\epsilon|$. The conclusion is that ${p{(\G_E)}p}$ is the infinite cyclic group generated by $(p,|\epsilon|,p)$.
	
	(\ref{GE isotropy2}) Let $(p,k,p) \in {p{(\G_E)}p}$. Then $p \sim_k p$ implies $p = \alpha x = \beta x$ for some $(\alpha, \beta) \in E^\star \times_r E^\star$ and $x \in r(\alpha) \partial E$, with $|\alpha| - |\beta| = k$. If $p$ is finite, this implies $\alpha = \beta$, so $k = 0$. That is, the isotropy group at $p$ is trivial. On the other hand, suppose $p$ is infinite and not eventually periodic. If $|\alpha|<|\beta|$, then $\beta = \alpha \beta'$ for some $\beta' \in E^\star$. But then $p = \alpha x = \beta x = \alpha \beta' x$, so $x = \beta'  x = \beta' \beta' x = \beta'\beta'\beta'\dots$, and this proves $p$ is eventually periodic, a contradiction. Similarly, assuming $|\beta| < |\alpha|$ reaches the same contradiction. Therefore, $|\alpha|= |\beta|$ and $k = 0$, implying that the isotropy group at $p$ is trivial.
\end{proof}

%

The next step is to define a topology on $\G_E$. Let $(\alpha, \beta) \in E^\star \times_r E^\star$, and let $F \subseteq_{\rm finite} r(\alpha)E^1$. Define the sets: \index{$\Z(\alpha, \beta)$, $\Z(\alpha, \beta, F)$}
\begin{align*} \label{basic sets def}
&\Z(\alpha, \beta) = \big\{(\alpha x, |\alpha| - |\beta|, \beta x) \mid x \in r(\alpha)\partial E \big\};&& \notag
\Z(\alpha, \beta, F) = \Z(\alpha, \beta) \setminus \bigcup_{e \in F} \Z(\alpha e, \beta e).
\end{align*}
Obviously, $\Z(\alpha, \beta) = \Z(\alpha, \beta, \emptyset)$. Next we present a pair of technical lemmas (generalising \cite[Lemma 2.5]{kumjian1997graphs}) which prove that the collection of sets of the form $\Z(\alpha, \beta, F)$ is closed under pairwise intersections, so it can serve as a base for a topology on $\G_E$.

\begin{lemma} \label{int lemma} 
	Let $(\alpha, \beta), (\gamma, \delta) \in E^\star \times_r E^\star$. Then
	\begin{align*}
	\Z(\alpha, \beta) \cap \Z(\gamma, \delta) =
	\begin{cases}
	\Z(\alpha, \beta) & \text{if } \exists\ \kappa \in E^\star,\ \alpha = \gamma\kappa,\ \beta = \delta\kappa\\
	\Z(\gamma, \delta) & \text{if } \exists\ \kappa \in E^\star,\ \gamma = \alpha\kappa,\ \delta = \beta\kappa\\
	\emptyset &\text{otherwise}
	\end{cases}
	\end{align*}
\end{lemma}

\begin{proof} 
	We prove that when the intersection of the two sets is nonempty, then it must be one of the first two cases in the piecewise expression.
	To this end, let $(\alpha x,|\alpha|-|\beta|,\beta x) = (\gamma x', |\gamma|-|\delta|, \delta x') \in \Z(\alpha, \beta) \cap \Z(\gamma,\delta)$, where $x \in r(\alpha)\partial E$ and $x' \in r(\gamma)\partial E$. Assume $|\gamma| \le |\alpha|$, which implies $|\delta| \le |\beta|$; if not, rearrange. Since $\alpha x = \gamma x'$, it must be that $\alpha = \gamma \kappa$ where $\kappa$ is the initial subpath of $x'$ of length $|\alpha|-|\gamma|$. Similarly, $\beta = \delta \kappa$. So we are in the first case (or the second case, if a rearrangement took place).
	In the first two cases in the piecewise expression, it is clear from the definitions what the intersection of $\Z(\alpha, \beta)$ and $\Z(\gamma, \delta)$ must be.
\end{proof}

\begin{lemma} \label{int lemma 2}
	Suppose $(\alpha, \beta), (\gamma, \delta) \in E^\star \times_r E^\star$, $F \subseteq_{\rm finite} r(\alpha)E^1$, and $H\subseteq_{\rm finite} r(\gamma)E^1$.
	Then
	\begin{equation*}
	\Z(\alpha, \beta, F) \cap \Z(\gamma, \delta, H)
	= \begin{cases}
	\Z(\alpha, \beta, F \cup H) &\!\!\text{if } \alpha = \gamma, \beta = \delta\\
	\Z(\alpha, \beta, F) & \!\!\text{if } \exists\ \kappa \in E^\star,\  |\kappa| \ge 1,\ \alpha = \gamma\kappa,\ \beta = \delta\kappa,\  \kappa_1 \notin H \\
	\Z(\gamma, \delta, H) & \!\!\text{if } \exists\ \kappa \in E^\star,\ |\kappa|\ge1,\ \gamma = \alpha\kappa,\ \delta = \beta\kappa,\ \kappa_1 \notin F \\
	\emptyset & \!\!\text{otherwise}
	\end{cases}
	\end{equation*}
\end{lemma}

\begin{proof} We make a calculation and then proceed by cases:
	\begin{align} \label{hugeeq}
	\Z(\alpha, \beta, F) \cap \Z(\gamma, \delta, H) 
	&= \left[
	\Z(\alpha, \beta) \setminus \bigcup_{e \in F} \Z(\alpha e, \beta e)
	\right]
	\bigcap
	\left[
	\Z(\gamma, \delta) \setminus \bigcup_{e \in H} \Z(\gamma e, \delta e)
	\right]
	\\ \notag
	&= \left[
	\Z(\alpha, \beta) \cap \Z(\gamma, \delta)
	\right]
	{ \setminus}
	\left[
	\bigcup_{e \in F} \Z(\alpha e, \beta e) \cup \bigcup_{e \in H} \Z(\gamma e, \delta e)
	\right].
	\end{align}
	\textit{Case 1: } If $\alpha = \gamma$ and $\beta = \delta$, equation (\ref{hugeeq}) yields $\Z(\alpha, \beta, F) \cap \Z(\gamma, \delta, H) = \Z(\alpha, \beta, F \cup H)$.
	
	\textit{Case 2: } If there exists $\kappa \in E^\star \setminus E^0$ such that $\alpha = \gamma \kappa$ and $\beta = \delta \kappa$ then after applying Lemma \ref{int lemma}, the right hand side of (\ref{hugeeq}) becomes
	\[
	\Z(\alpha, \beta) \setminus
	\left[
	\bigcup_{e \in F} \Z(\alpha e, \beta e) \cup \bigcup_{e \in H}\Z(\gamma e, \delta e)
	\right].
	\]
	Moreover, $\Z(\alpha, \beta) \cap \Z(\gamma e , \delta e) = \emptyset$ for all $e \in H$, provided $ e \ne \kappa_1$. If $e = \kappa_1$ then $\Z(\alpha, \beta) \cap \Z(\gamma e, \delta e) = \Z(\alpha, \beta)$. Therefore (\ref{hugeeq}) becomes $\Z(\alpha, \beta, F)$ if $\kappa_1 \notin H$ and $\emptyset$ if $\kappa_1 \in H$.
	
	\textit{Case 3: } If there exists $\kappa \in E^\star\setminus E^0$ such that $\gamma = \alpha \kappa$ and $\delta = \beta \kappa$ then the situation is symmetric to the second case.
	
	\textit{Case 4: } Otherwise, $\Z(\alpha, \beta) \cap \Z(\gamma, \delta) = \emptyset$, by Lemma  \ref{int lemma}.
\end{proof}

From now on, we assume $\G_E$ has the topology generated by all the sets: 
\begin{align} \label{Z-base}
\Z(\alpha, \beta, F);&& (\alpha, \beta) \in E^\star\times_r E^\star,\ F 	\subseteq_{\rm finite} r(\alpha)E^1.
\end{align}

Some of our references give a different base for the topology on $\G_E$, but all the different bases that we know of contain the sets $\Z(\alpha, \beta, F)$. There are advantages to working with a base that is not too large, which is why we have chosen to focus on this one.

Let $E$ be a graph and consider $\mathbb{Z}$ with the discrete topology. The map
\begin{align*}
\theta: \G_E \to \mathbb{Z},&& (x, k, y) \mapsto k,
\end{align*}
is a continuous groupoid homomorphism. In fact, it is a degree map giving $\G_E$ the structure of a $\ZZ$-graded groupoid. Some parts of this lemma are reminiscent of \cite[Proposition 2.6]{kumjian1997graphs}.

\begin{lemma} \label{grp top}
	Let $E$ be a graph.
	\begin{enumerate}[\rm(1)]
		\item \label{grp top 1}
		The topology on $\G_E$ is Hausdorff.
		\item \label{grp top 2}
		$\dom: \G_E \to \partial E$ is a local homeomorphism.
		\item \label{grp top 3}
		If $(\alpha, \beta) \in E^\star\times_r E^\star$ and $F \subseteq_{\rm finite} r(\alpha)E^1$, then $\Z(\alpha, \beta, F)$ is compact.
	\end{enumerate}
\end{lemma}

\begin{proof}
	(\ref{grp top 1}) Take $(x,k,y) \ne (w,\ell,z)$ in $\G_E$. If $k \ne \ell$ then $\theta^{-1}(k)$ and $\theta^{-1}(\ell)$ are disjoint open sets separating the two points. Otherwise, either $x \ne w$ or $y \ne z$. If $w \ne x$ then either: $w$ and $x$ must differ on some initial segment, or one must be an initial subpath of the other. Using Lemma \ref{int lemma 2}, it is not difficult to separate the two points by disjoint open sets. If $y \ne z$, the same reasoning applies.
	
	(\ref{grp top 2})	For $(\alpha, \beta) \in E^\star \times_r E^\star$, define
	\begin{align*}
	h_{\alpha, \beta}: Z(\beta) \to \Z(\alpha, \beta), && \beta x \mapsto  (\alpha x, |\alpha|-|\beta|, \beta x).
	\end{align*}
	Clearly, $h_{\alpha, \beta}$ is a bijection. By Lemma \ref{int lemma 2}, the basic open sets contained in $\Z(\alpha, \beta)$ are all of the form $\Z(\alpha\kappa,\beta \kappa, F')$ where $\kappa \in r(\alpha)E^\star$ and $F' \subseteq_{\rm finite} r(\kappa)E^1$. Clearly
	\[h_{\alpha, \beta}^{-1}\big(\Z(\alpha\kappa,\beta \kappa, F')\big) = Z(\beta\kappa, F')
	\]
	is open in $Z(\beta)$, so $h_{\alpha, \beta}$ is continuous. A continuous map from a compact space to a Hausdorff space is a closed map, so $h_{\alpha, \beta}$ is a closed map. Therefore $h_{\alpha, \beta}$ is a homeomorphism. This proves that $\dom|_{\Z(\alpha, \beta)}$ is a homeomorphism onto its image (because $\dom|_{\Z(\alpha, \beta)}^{-1} = h_{\alpha, \beta}$).
	
	(\ref{grp top 3}) According to item (\ref{grp top 2}), $\dom$ restricts to a homeomorphism $\Z(\alpha, \beta, F)\approx Z(\beta, F)$, and $Z(\beta, F)$ is compact by Theorem \ref{path space topology}.
\end{proof}

Since $\Z(\alpha, \beta, F) \approx Z(\beta, F)$, Lemma \ref{non-empty} implies that $\Z(\alpha, \beta, F) = \emptyset$ if and only if $r(\alpha)$ is a regular vertex and $F = r(\alpha)E^1$.

\begin{remark}
	The groupoid $\G_E$ admits continuous maps
	\begin{align*}
	&\cod: (x,k,y)\mapsto x, & &\theta: (x,k,y)\mapsto k,& &\dom: (x,k,y) \mapsto y,
	\end{align*}
	so it is tempting to think that the topology on $\G_E$ coincides with the relative topology that it gets from being a subset of the product space $\partial E \times \ZZ \times \partial E$. However, this is not the case: the topology on $\G_E$ is much finer than the relative topology from $\partial E \times \ZZ \times \partial E$.
\end{remark}

The main theorem that follows is not new, and it has been in use for some time. Indeed, it is implied by \cite[Lemma 2.1]{renault2018uniqueness}, although not in a trivial way (see also \cite[Theorem 3.5]{paterson2002graph} and \cite[Theorem 3.16]{yeend2007groupoid}). However, this is the first self-contained proof that we know of that applies to ordinary directed graphs, and does not require the graph to be countable. 

\begin{theorem} \label{weeks}
	Let $E$ be a graph. The groupoid $\G_E$ is a Hausdorff ample groupoid with the base of compact open bisections given in (\ref{Z-base}).
\end{theorem}

\begin{proof}
	The most technical part that remains is showing that the composition map $\mult$ is continuous. If $x, z \in E^\star \cap \partial E$ are tail equivalent finite paths, then $(x, |x|-|z|, z)$ has a neighbourhood base of open sets,
	$
	\mathcal{N}_{(x,|x|-|z|,z)} = \{\Z(x, z, F) \mid F \subseteq_{\rm finite} r(x)E^1\}
	$.
	If $x, z \in E^\infty$ are tail equivalent infinite paths, with lag $t$, then there exists $N \ge 0$ such that $\sigma^{N+t}(x) = \sigma^{N}(z)$. Consequently $(x, t, z)$ has a neighbourhood base of open sets, 
	$\mathcal{N}_{(x,t,z)} = \{\Z(x_1 \dots x_{n+t}, z_1 \dots z_n) \mid n > N\}$.
	
	Now suppose $U$ is an open set in $\G_E$ containing a product of two morphisms $(x, k+\ell, z) = (x,k,y)(y,\ell,z)$. It must be that $x,y,z$ are all finite paths or they are all infinite paths. If $x,y,z$ are finite paths, then they must have $r(x) = r(y) = r(z)$ and $U$ must contain some $\Z(x,z,F) \in \mathcal{N}_{(x,|x|-|z|,z)}$. Then $\big((x, k, y),(y,\ell,z)\big)$ is contained in the open set $\big(\Z(x,y,F) \times \Z(y,z,F)\big) \cap \G_E^{(2)}$ which is mapped bijectively by $\mult$ into $Z(x,z,F) \subseteq U$. Otherwise $x,y,z$ are all infinite paths, and there must exist $n$ large enough that $\sigma^{n+k+\ell}(x) = \sigma^{n+\ell}(y) = \sigma^{n}(z)$. Making $n$ even larger if necessary, we can assume $U$ contains some $\Z(x_1 \ldots x_{n+k+\ell}, z_1 \ldots z_{n}) \in \mathcal{N}_{(x,k+ \ell,z)}$. Define:
	\begin{align*}
	x' = x_1 \ldots x_{n+k+\ell}, &&
	y' = y_1 \ldots y_{n+\ell}, &&
	z' = z_1 \ldots z_{n}.
	\end{align*}
	Then $\big((x,k,y),(y,\ell,z)\big)$ is contained in the open set $\big(\Z(x',y') \times \Z(y', z')\big) \cap \G_E^{(2)}$, which is mapped bijectively by $\mult$ into $Z(x', z') \subseteq U$. Since $(x, k+\ell, z) = (x,k,y)(y,\ell,z)$ was an arbitrary product in $U$, this shows that $\mult^{-1}(U)$ is open in $\G_E^{(2)}$, so $\mult$ is continuous.
	It is much easier to show that the inversion map $\inv$ is continuous, because $\inv$ puts $Z(\alpha, \beta,F)$ in bijection with $Z(\beta, \alpha, F)$. We have proved $\G_E$ is a topological groupoid.
	In Lemma \ref{grp top}~(\ref{grp top 2}), it is shown that $\dom$ is a local homeomorphism. Therefore, $\G_E$ is an \'etale groupoid. The remaining facts from Lemma \ref{grp top} establish that $\G_E$ is a Hausdorff ample groupoid and that the base described in $(\ref{Z-base})$ consists of compact open bisections.
\end{proof}

\section{The Leavitt path algebra of a graph} \label{CHAPTER: LEAVITT PATH ALGEBRAS}

In \S\ref{LPA intro}, we define the Leavitt path algebra of a graph. We define it in terms of its universal property, and then describe how it can be realised as the quotient of a path algebra. Path algebras are, in some sense, the definitive examples of $\ZZ$-graded algebras, and the $\ZZ$-grading survives in their Leavitt path algebra quotients. In \S\ref{uniq}, we prove the Graded Uniqueness Theorem for Leavitt path algebras. In \S\ref{steinberg model}, we prove the cornerstone result: the Leavitt path algebra of a graph is isomorphic to the Steinberg algebra of its boundary path groupoid. Through this lens, we rederive some fundamentals of Leavitt path algebras, and classify finite-dimensional Leavitt path algebras. In \S\ref{steinberg uniq}, we prove the Graded and Cuntz-Krieger Uniqueness Theorems for Steinberg algebras and use them to prove the Cuntz-Krieger Uniqueness Theorem for Leavitt path algebras.

\begin{remark}
Historically, the theory of Leavitt path algebras was developed for the case when $R$ is a field, and $E$ is a row-finite countable graph. Later, the methods were improved and $R$ could be any unital commutative ring if $E$ is a countable graph \cite{abrams2008leavitt,tomforde2011leavitt}. Alternatively, $E$ could be an arbitrary graph if $R$ is a field \cite{LPAbook,goodearl2009leavitt}. The proofs of some key results, including the fact that the relations on $L_R(E)$ do not collapse the algebra to zero (\cite[Lemma 1.5]{goodearl2009leavitt} and \cite[Proposition 3.4]{tomforde2011leavitt}) and the Graded Uniqueness Theorem (\cite[Proposition 3.6]{goodearl2009leavitt} and \cite[Theorem 5.3]{tomforde2011leavitt}), have not yet been recorded for the case where simultaneously $E$ is uncountable and $R$ is not a field. Here, we fix this and complete the picture.
\end{remark}

\subsection{Introducing Leavitt path algebras} \label{LPA intro}

Let $E = (E^0, E^1, r,s)$ be a graph. We introduce the set of formal symbols \index{$(E^1)^*$}
$
(E^1)^* = \{e^*\mid e \in E^1\}
$
and call the elements of $(E^1)^*$ \textit{ghost edges}. For clarity, we will sometimes refer to the elements of $E^1$ as \textit{real edges}. If $\alpha = \alpha_1 \dots \alpha_{|\alpha|} \in E^\star$ is a finite path of positive length, we define $\alpha^*$ to be the sequence $\alpha_{|\alpha|}^* \dots \alpha_1^*$, and call it a \textit{ghost path}. We also define $v^* = v$ for every $v \in E^0$.

\begin{definition}\cite{tomforde2011leavitt} \label{E-family}
	Let $E$ be a graph and let $A$ be a ring. Assume $\{v, e, e^*\mid v \in E^0, e \in E^1 \}$ is a subset of $A$; in other words, there is a function $E^0 \sqcup E^1 \sqcup (E^1)^* \to A$ whose image inherits the notation of its domain. Then $\{v, e, e^*\mid v \in E^0, e \in E^1 \}\subset A$ is called a \textbf{Leavitt $E$-family} if the following conditions are satisfied:
	\begin{enumerate}[(V)]
		\item \label{V}
		$v^2 = v$ and $vw = 0$ for all $v, w \in E^0$, $v \ne w$;
	\end{enumerate}
	\begin{enumerate}[(E1)]
		\item \label{E1}
		$s(e) e = e r(e) = e$ for all $e \in E^1$;
		\item \label{E2}
		$e^*s(e) = r(e)e^* = e^*$ for all $e \in E^1$;
	\end{enumerate}
	\begin{enumerate}[(CK1)]
		\item \label{CK1}
		$e^* e = r(e)$ and $e^*f = 0$ for all $e, f \in E^1$, $e \ne f$;
		\item \label{CK2}
		$v = \sum_{e \in vE^1} e e^*$ for all $v \in E^0_{\rm reg}$.
	\end{enumerate}
\end{definition}
The interpretation of (\hyperref[V]{V}) is that $\{v \in A \mid v \in E^0 \}$ is a set of pairwise orthogonal idempotents. The relations (\hyperref[CK1]{CK1}) and (\hyperref[CK2]{CK2}) are called the Cuntz-Krieger relations, and they originate from operator theory. The relevant interpretation, at least in that setting, is that vertices are represented by projections, and edges are represented by partial isometries with mutually orthogonal ranges.

In any algebra $A$ containing a Leavitt $E$-family $\{v, e, e^*\mid v \in E^0, e \in E^1 \}$, one can consider paths $\mu =  \mu_1 \dots \mu_{|\mu|}$ and ghost paths $\mu^* = \mu_{|\mu|}^* \dots \mu_1^*$ as elements of $A$ in the obvious way: products of their constituent real edges and ghost edges respectively. The following lemma is straightforward to prove using the relations (\hyperref[E1]{E1}), (\hyperref[E2]{E2}), and (\hyperref[CK1]{CK1}). It is so fundamental that we will usually use the result without referring to it.

\begin{lemma}\label{product of monomials}
	If $A$ is an $R$-algebra generated by a Leavitt $E$-family $\{v,e,e^* \mid v \in E^0, e \in E^1\}$, the elements of $A$ obey the rule:
	\[
	(r\mu \nu^*)(r'\gamma \lambda^*)
	= \begin{cases} 
	(rr')\mu \kappa^* \lambda^* & \text{if }\gamma\text{ is an initial subpath of } \nu\text{, with }\nu = \gamma \kappa\\
	(rr')\mu \kappa \lambda^* & \text{if }\nu\text{ is an initial subpath of } \gamma\text{, with }\gamma = \nu\kappa \\
	0 & \text{otherwise}
	\end{cases}
	\]
	for all $r, r' \in R$ and all $\mu, \nu, \gamma, \lambda \in E^\star$, with $r(\mu) = r(\nu)$ and $r(\gamma) = r(\lambda)$.
\end{lemma}

\begin{corollary}
	\label{monomials} Every $R$-algebra generated by a Leavitt $E$-family is generated, as an abelian group, by the set $\{r\alpha\beta^*\mid r \in R,(\alpha, \beta) \in E^\star\times_r E^\star\}$.
\end{corollary}

\begin{proof}
	By Lemma \ref{product of monomials}, every word in the generators $\{v,e,e^* \mid v \in E^0, e \in E^1 \}$ reduces to an expression of the form $\alpha\beta^*$ where $\alpha, \beta \in E^\star$. Moreover, $\alpha \beta^* = 0$ unless $r(\alpha) = r(\beta)$, by (\hyperref[V]{V}), (\hyperref[E1]{E1}), and (\hyperref[E2]{E2}).
\end{proof}

Let $B$ be an $R$-algebra generated by a Leavitt $E$-family $\{v,e,e^* \mid v \in E^0, e \in E^1\}$. We say that $B$ is \textit{universal} (for Leavitt $E$-families) if every $R$-algebra ${A}$ containing a Leavitt $E$-family $\{a_v, b_e, c_{e^*}\mid v \in E^0, e \in E^1 \}$ admits a unique $R$-algebra homomorphism $\pi: B \to {A}$ such that $\pi(v) = a_v$, $\pi(e) = b_e$, and $\pi(e^*) = c_{e^*}$ for every $v \in E^0$ and $e \in E^1$. The universal property determines $B$ up to isomorphism.

\begin{definition}\label{LPA def} \index{$L_R(E)$}
	Let $E$ be a graph. The \textbf{Leavitt path algebra} of $E$ with coefficients in $R$, denoted by $L_R(E)$, is the universal $R$-algebra generated by a Leavitt $E$-family.
\end{definition}

Technically, $L_R(E)$ is an isomorphism class in the category of $R$-algebras. If $B$ is a specific $R$-algebra having the universal property for Leavitt $E$-families, then $B$ is a \textit{model} of $L_R(E)$. However, it is customary and natural to refer to $L_R(E)$ as if it were a specific model with the standard generators $\{v,e,e^* \mid v \in E^0, e \in E^1\}$. Every element $x \in L_R(E)$, so to speak, is a finite sum of the form $x = \sum r_i \alpha_i \beta_i^*$ where $r_i \in R$ and $(\alpha, \beta) \in E^\star \times_r E^\star$ for all $i$. Such an expression for $x$ is not necessarily unique, owing to the (\hyperref[CK2]{CK2}) relation. If we have reason to consider a different model of $L_R(E)$, say another $R$-algebra $B$, then we would write $L_R(E) \cong B$.

\begin{examples}\cite[\S1.3]{LPAbook} \label{graph examples} Sometimes $L_R(E)$ can be recognised as a more familiar algebra. Four fundamental examples of Leavitt path algebras are:
	\begin{enumerate}[(a)]
		\item 	The \textit{finite line graph} with $n$ vertices is the graph pictured below:
		\begin{align*}
		A_n \quad &= \xymatrix{&\bullet^{v_1} \ar[r]^{e_1} & \bullet^{v_2} \ar[r]^{e_2} & \bullet^{v_3} \ar@{..}[r] & \bullet^{v_{n-1}} \ar[r]^{e_{n-1}} & \bullet^{v_n} }
		\end{align*}
		It turns out that $L_R(A_n) \cong M_n(R)$, the \textbf{matrix algebra} of $n \times n$ matrices over $R$. Explicitly, the set of standard matrix units $\{E_{i,j}\mid 1 \le i,j \le n\} \subset M_n(R)$ contains a Leavitt $E$-family $\{a_v, b_{e}, c_{e^*} \mid v \in A_n^0, e \in A_n^1\}$, where:
		\begin{align*}a_{v_i} = E_{i,i},& & b_{e_j} = E_{j,j+1}, && c_{e_j^*} = E_{j+1,j},
		&& 1 \le i \le n,\ 1 \le j \le n-1.
		\end{align*}
		\item 		\label{rose with n petals graph}
		The \textit{rose with $n$ petals} is the graph pictured below (see also Example \ref{rosetwo}):
		\begin{align*}
		R_n \quad &=  \xymatrix{ &{\bullet^v} \ar@(ur,dr) ^{e_1} \ar@(u,r) ^{e_2} \ar@(ul,ur) ^{e_3} \ar@{.} @(l,u) \ar@{.} @(dr,dl)
			\ar@(r,d) ^{e_n} \ar@{}[l] ^{\ldots} }
		\end{align*}
		
		The Leavitt path algebra $L_R(R_n)$ is isomorphic to the \textbf{Leavitt algebra} $L_{n,R}$,
		discovered by W. G. Leavitt in \cite[\S3]{leavitt1962module}. It is from this example that the Leavitt path algebras get their name.
		\item The \textit{rose with 1 petal},
		\begin{align*}
		R_1 \quad &= \xymatrix{&{^v\bullet} \ar@(ur,dr)^e}
		\end{align*}
		gives rise to the algebra of \textbf{Laurent polynomials} $R[x,x^{-1}]$.
		\item The \textit{Toeplitz graph},
		\begin{align*}
		T \quad &= \xymatrix{ &{^u\bullet} \ar@(ul,dl)_e \ar[r]^f & \bullet^v }
		\end{align*}
		gives rise to the \textbf{Toeplitz $R$-algebra}, which has the presentation $R\langle x,y \mid xy = 1 \rangle$. The isomorphism $R\langle x, y  \mid xy = 1 \rangle \to L_K(T)$ maps $x \mapsto e^* + f^*$ and $y \mapsto e + f$.
	\end{enumerate}
\end{examples}

As an alternative to Definition \ref{LPA def}, it is popular to define the Leavitt path algebra of a graph as a certain quotient of a path algebra. The path algebra of a graph (also called the quiver algebra of a quiver) is an older concept, familiar to a wider audience of algebraists and representation theorists. We have defined $L_R(E)$ by its universal property, so we look towards path algebras to provide a model of $L_R(E)$, thereby proving that $L_R(E)$ exists.

Let $E = (E^0, E^1, r, s)$ be a graph. The \textit{path algebra} of $E$ with coefficients in $R$ is the free $R$-algebra generated by $E^0 \sqcup E^1$, modulo the ideal generated by the relations (\hyperref[V]{V}) and (\hyperref[E1]{E1}).
The \textit{extended graph} of $E$ is defined as $\widehat{E}= (E^0, E^1 \sqcup(E^1)^*, r', s')$, \index{$\widehat{E}$} where $r'$ and $s'$ are extensions of $r$ and $s$, respectively:
\begin{align*}
r'(e) = r(e) \text{ for all } e \in E^1; &&r'(e^*) = s(e) \text{ for all } e^* \in (E^1)^* \\
s'(e) = s(e) \text{ for all } e \in E^1; &&s'(e^*) = r(e) \text{ for all } e^* \in (E^1)^*.
\end{align*}
In other words, $\widehat E$ is formed from $E$ by adding a new edge $e^*$ for each edge $e$, such that $e^*$ has the opposite direction to $e$. The path algebra $R\widehat E$ can be characterised as the free $R$-algebra generated by $E^0 \sqcup E^1 \sqcup (E^1)^*$, subject to the relations (\hyperref[V]{V}), (\hyperref[E1]{E1}), and (\hyperref[E2]{E2}).
Let $\mathcal A$ be the quotient of $R\widehat{E}$ by the ideal generated by the relations (\hyperref[CK1]{CK1}) and (\hyperref[CK2]{CK2}). By virtue of its construction, $\mathcal A$ has the universal property for Leavitt $E$-families, and consequently $\mathcal A \cong L_R(E)$. The path algebra model is useful for proving the following fact.

\begin{proposition}
	The Leavitt path algebra $L_R(E) = \bigoplus_{n \in \ZZ}L_R(E)_n$ is a $\mathbb{Z}$-graded algebra, where the homogeneous components are: 
	\[
	L_R(E)_n = \Span_R
	\left\{
	\mu \nu^*\mid(\mu, \nu) \in E^\star\times_r E^\star, |\mu| - |\nu| = n
	\right\}.
	\]
\end{proposition}

\begin{proof}
	Naturally, the free $R$-algebra $R\langle E^0 \cup E^1 \cup (E^1)^*\rangle$ is $\mathbb{Z}$-graded by setting $\mathrm{deg}(v) = 0$ for all $v \in E^0$, and $\mathrm{deg}(e) = 1$, $\mathrm{deg}(e^*) = -1$ for all $e \in E^1$. Extending the degree map (in the only possible way) yields $\mathrm{deg}(a_1 \dots a_n) = \sum_{i = 1}^n \mathrm{deg}(a_n)$ for any word $a_1 \dots a_n \in R\langle E^0 \cup E^1 \cup (E^1)^*\rangle $. The relations (\hyperref[V]{V}), (\hyperref[E1]{E1}), and (\hyperref[E2]{E2}) are all homogeneous with respect to the grading on $R\langle E^0 \cup E^1 \cup (E^1)^*\rangle$, so they generate a graded ideal, and the quotient $R \widehat E$ is $\mathbb{Z}$-graded. Similarly, relations (\hyperref[CK1]{CK1}) and (\hyperref[CK2]{CK2}) are homogeneous with respect to the grading on $R\widehat E$, so they generate a graded ideal, and the quotient $L_R(E)$ is $\ZZ$-graded. The word $\mu\nu^*$ has degree $|\mu| - |\nu|$ in $R\langle E^0 \cup E^1 \cup (E^1)^*\rangle$, which gives the expression for the homogeneous components of $L_R(E)$.
\end{proof}

\subsection{Uniqueness theorems for Leavitt path algebras} \label{uniq}
Research on graph algebras has made extensive use of two main kinds of uniqueness theorems: the Cuntz-Krieger uniqueness theorems, and the graded uniqueness theorems. (In the analytic setting, graded uniqueness theorems are replaced by gauge invariant uniqueness theorems.) These theorems give sufficient conditions for a homomorphism to be injective, so they are very useful for establishing isomorphisms between a graph algebra and another algebra that comes from somewhere else. They are also very useful for studying structural properties like primeness and simplicity. Appropriate versions of these theorems have been proved not just for Leavitt path algebras but also (and we refer to \cite{clark2017cohn,clark2017kumjian,raeburn2005graph,renault2018uniqueness}) for graph $C^*$-algebras, as well as Cohn path algebras, higher-rank graph algebras, and even algebras of topological higher-rank graphs. 

This section provides a brief account of the uniqueness theorems for Leavitt path algebras. For the Graded Uniqueness Theorem, we adhere to Tomforde's proof from \cite{tomforde2011leavitt}.

\begin{lemma}\cite[Lemma 5.1]{tomforde2011leavitt} \label{ideal generation}
	Let $I$ be a graded ideal of $L_R(E)$, where $E$ is a graph. Then $I$ is generated as an ideal by its $0$-component $I_0 = I \cap L_R(E)_0$.
\end{lemma}

\begin{proof}
	Since $I$ is a graded ideal, $I = \sum_{k \in \ZZ} I_k$, where $I_k = I \cap L_R(E)_k$. Let $k>0$ and $x \in I_k$. By Corollary \ref{monomials}, we can write $x = \sum_{i = 1}^n \alpha_i x_i$ where each $x_i \in L_R(E)_0$, and each $\alpha_i \in E^\star$ is distinct with $|\alpha_i| = k$. Then for $1 \le j \le n$, we have $x_j = \alpha_j^*\left(\sum_{i = 1}^n \alpha_i x_i\right) = \alpha_j^* x \in I_0$. So, $I_k$ is spanned by elements of the form $\alpha_j x_j$ where $\alpha_j \in L_R(E)_k$ and $x_j \in I_0$.	That is, $I_k = L_R(E)_k I_0$.
	Similarly, if $y \in I_{-k}$ then we can write $y = \sum_{i = 1}^m y_i \beta_i^*$ where each $y_i \in L_R(E)_0$, and each $\beta_i  \in E^\star$ is distinct with $|\beta_i| = k$. Then for $1 \le j \le n$, we have $y_j = \left(
	\sum_{i = 1}^n y_i \beta_i^*
	\right) \beta_j = y \beta_j \in I_0$. Therefore $I_{-k}$, is spanned by elements of the form $y_j \beta_j$ where $\beta_j \in L_R(E)_{-k}$ and $y_j \in I_0$. That is, $I_{-k} = I_0 L_R(E)_{-k}$. Since $I = \sum_{n \in \mathbb{Z}}I_k$, this shows $I$ is the ideal generated by $I_0$.
\end{proof}

The next lemma is a slight variation of the Reduction Theorem \cite[Theorem 2.2.11]{LPAbook}. The lemma needs the assumption that $rv \in L_R(E)$ is nonzero for every $r \in R\setminus \{0\}$ and $v \in E^0$. In fact, this is always true, but we shall only prove it later.

\begin{lemma}\cite[Lemma 5.2]{tomforde2011leavitt}.\label{reduction}
	Let $E$ be an arbitrary graph. Assume $rv \in L_R(E)$ is nonzero for every $r \in R \setminus \{0\}$ and $v \in E^0$. If $x \in L_R(E)_0$ is nonzero, then there exists $(\alpha, \beta) \in E^\star\times_r E^\star$ and $s \in R \setminus \{0\}$, such that $\alpha^* x \beta = s r(\alpha)$.
\end{lemma}

\begin{proof}
	The set $\mathcal{M}_n = \Span_R\{ \alpha \beta^* \mid 1 \le |\alpha| = |\beta| \le n \}$ is an $R$-submodule of $L_R(E)_0$, and indeed $L_R(E)_0 = \bigcup_{n = 0}^\infty \mathcal{M}_n$. The strategy is to prove inductively that for all $n \ge0$ the claim holds: for all $0 \ne x \in \mathcal{M}_n$ there exists $(\alpha, \beta) \in E^\star \times_r E^\star$ and $s \in R\setminus \{0\}$ such that $\alpha^* x \beta = sr(\alpha)$. The base case is $n = 0$. If $x \in \mathcal{M}_0$ then $x$ is a linear combination of vertices. Say $x = \sum_i r_i v_i$ with the $v_i$ being distinct vertices and the $r_i\in R\setminus \{0\}$. Then $v_1 x v_1 = r_1 v_1$ proves the claim. Now assume the claim holds for $n-1$. Let $0 \ne x \in \mathcal{M}_n$. We can write
	\begin{equation} \label{x}
	x = \sum_{i = 1}^p r_i \alpha_i \beta_i^* + \sum_{j = 1}^q s_j v_j
	\end{equation}
	where for all $1 \le i \le p$ and all $1 \le j \le q$: $r_i, s_j \in R\setminus \{0\}$, $(\alpha_i, \beta_i) \in E^\star \times_r E^\star$ with $1 \le |\alpha_i| = |\beta_i| \le n$, and $v_j \in E^0$. Further assume that all the $(\alpha_i, \beta_i)$ are distinct and all the $v_j$ are distinct. In the first case, if $v_j$ is a sink for some $1 \le j \le q$, then $v_j x v_j = s_j v_j$ proves the claim. In the second case, if $v_j$ is an infinite emitter for some $1 \le j \le q$, then there is an edge $e \in v_j E^1 \setminus \{(\alpha_{1})_1, \dots, (\alpha_p)_1\}$ and $e^*xe = s_j r(e)$ proves the claim. Otherwise, in the third case, every $v_j$ is a regular vertex. Applying (\hyperref[CK2]{CK2}), it is possible to expand $v_j = \sum_{e \in v_jE^1} ee^*$ for all $1 \le j \le q$. Then (\ref{x}) can be rewritten as   
	\begin{equation} \label{x2}
	x = \sum_{i = 1}^p t_i e_i\mu_i \nu_i^*f_i^*
	\end{equation}
	where $t_i \in R\setminus \{0\}$, $e_i, f_i \in E^1$, and $(e_i\mu_i, f_i\nu_i) \in E^\star \times_r E^\star$ for all $i \le 1 \le p$. It is safe to assume that
	\[\sum_{\substack{1 \le j \le p\\ e_j = e_1, f_j = f_1}}
	t_je_j\mu_j \nu_j^*f_j^*
	= e_1\Bigg(\sum_{\substack{1 \le j \le p\\ e_j = e_1, f_j = f_1}}
	t_j\mu_j \nu_j^*\Bigg)f_1^* \ne 0,
	\]
	otherwise it could just be removed from the sum in (\ref{x2}). Then, define
	\[
	x' = \sum_{\substack{1 \le j \le p\\ e_j = e_1, f_j = f_1}}
	t_j \mu_j \nu_j^*,
	\]
	noting that $0\ne x'\in \mathcal{M}_{n-1}$. By the inductive assumption there exists $(\alpha, \beta) \in E^\star \times_r E^\star$ and $s \in R \setminus \{0\}$ such that $\alpha^* x' \beta = sr(\alpha)$. Clearly $x' = e_1^* x f_1$, so $\alpha^* x' \beta = \alpha^* e_1^* x f_1 \beta = sr(\alpha)$. By assumption, $sr(\alpha) \ne 0$; this implies $e_1 \alpha$ and $f_1 \beta$ are legitimate paths with the same range. The claim is now proved for $n$, and by mathematical induction it holds for all $n \ge 0$.
\end{proof}

Combining these lemmas proves the Graded Uniqueness Theorem for Leavitt path algebras. This generalises both \cite[Theorem 5.3]{tomforde2011leavitt} and \cite[Theorem 3.2]{goodearl2009leavitt} by removing any restrictions on the cardinality of $E$, and by not requiring $R$ to be a field. However, we emphasise that this is essentially Tomforde's proof with the insight that countability is not required.

\begin{theorem}[Graded Uniqueness Theorem for Leavitt path algebras] \label{gradeduniqueness} $\phantom{t}$
	
	Let $E$ be a graph, and $R$ a unital commutative ring. If $A$ is a $\mathbb{Z}$-graded ring and $\pi: L_R(E) \to A$ is a graded homomorphism with the property that $\pi(rv) \ne 0$ for every $v \in E^0$ and every $r \in R\setminus\{0\}$, then $\pi$ is injective.
\end{theorem}

\begin{proof}
	The first observation is that $rv \ne 0$ (because $\pi(rv) \ne 0$) for every $v \in E^0$ and $r \in R \setminus \{0\}$. The second observation is that $\ker  \pi $ is a graded ideal, because $\pi$ is a graded homomorphism. Suppose $x \in (\ker \pi)_0 = \ker \pi  \cap L_R(E)_0$. If $x \ne 0$, then by Lemma \ref{reduction}, there exists $(\alpha, \beta) \in E^\star\times_r E^\star$ and $s \in R\setminus\{0\}$, such that $\alpha^*x \beta = sr(\alpha)$. Then $\pi(sr(\alpha)) = \pi(\alpha^*)\pi(x)\pi(\beta) = 0$, which is a contradiction. Therefore $x = 0$, so $(\ker \pi)_0 = 0$. Lemma \ref{ideal generation} proves that $\ker \pi$ is generated as an ideal by $(\ker\pi)_0 = 0$; consequently, $\ker \pi = 0$, so $\pi$ is injective.
\end{proof}

\begin{corollary} \label{vertex in ideal}
	For every nonzero graded ideal $I$ of $L_R(E)$, there exists $r \in R \setminus \{0\}$ and $v \in E^0$ such that $rv \in I$.
\end{corollary}

In fact, all of the uniqueness theorems have a corollary of this sort. We will not always write it so explicitly. The Cuntz-Krieger Uniqueness Theorem is similar in spirit to the Graded Uniqueness Theorem. We do not require the homomorphism to be graded, this time, but pay the price of an extra condition on the graph, called Condition (L).

\begin{definition}
	A graph $E$ satisfies \textbf{Condition (L)} \label{L} if every cycle has an exit.
\end{definition}

Note that $E$ satisfies Condition (\hyperref[L]{L}) if and only if every closed path has an exit; this is fairly intuitive and it is proved in \cite[Lemma 2.5]{abrams2005leavitt}. Combining \cite[Theorem 6.5]{tomforde2011leavitt} and
\cite[Theorem 3.6]{goodearl2009leavitt} (see also \cite[Theorem 2.2.16]{LPAbook}) produces a version of the Cuntz-Krieger Uniqueness Theorem for Leavitt path algebras.

\begin{theorem} \label{early-CK-uniqueness}
	Let $E$ be a graph satisfying Condition (\hyperref[L]{L}) and let $R$ be a unital commutative ring, such that either $E$ is countable or $R$ is a field.
	If $A$ is a ring and $\psi: L_R(E) \to A$ is a homomorphism with the property that $\psi(rv) \ne 0$ for every $v \in E^0$ and every $r \in R \setminus \{0\}$, then $\psi$ is injective.
\end{theorem}

This theorem can be proved for a field $R= \mathbb{K}$, using the Reduction Theorem \cite[Theorem 2.2.11]{LPAbook}. However, we shall prove it later using groupoid methods instead. In doing so, we remove the awkward restrictions on $E$ and $R$.

\subsection{The Steinberg algebra model} \label{steinberg model}

Here, we prove the existence of a Steinberg algebra model for Leavitt path algebras, and use it to prove some fundamental facts.

\begin{theorem}\cite{clark2015equivalent} \label{steinberg-leavitt}
	Let $E$ be a graph and $R$ a unital commutative ring. Then $L_R(E)$ and $A_R(\G_E)$ are isomorphic as $\ZZ$-graded $R$-algebras.
\end{theorem}
\begin{proof}
	For $v \in E^0$ and $e \in E^1$, define
	\begin{align*}
	&a_v = \bm{1}_{Z(v)}, &&
	b_e = \bm{1}_{\Z(e,r(e))}, &&
	b_{e^*} = \bm{1}_{\Z(r(e),e)}.
	\end{align*}
	We can routinely validate that $\{a_v, b_e, b_e^*\mid v \in E^0, e \in E^1\}$ is a Leavitt $E$-family. For all $e, f \in E^1$, $v,w \in E^0$, and $u \in E^0_{\rm reg}$:
	\begin{gather*} 
	a_v a_w = \bm{1}_{Z(v)}\bm{1}_{Z(w)} = \bm{1}_{Z(v)\cap Z(w)} = \delta_{v,w}
	\bm{1}_{Z(v)}, \tag{V} \displaybreak[0] \\ 
	a_{s(e)}b_ea_{r(e)} = \bm{1}_{Z(s(e))\Z(e,r(e))Z(r(e))} = 
	\bm{1}_{\Z(e,r(e))} = b_e, \tag{E1} \displaybreak[0]
	\\	
	a_{r(e)}b_{e^*}a_{s(e)} = \bm{1}_{Z(r(e))\Z(r(e),e)Z(s(e))} = 
	\bm{1}_{\Z(r(e), e)} = b_{e^*} \tag{E2}, \displaybreak[0] \\
	b_{e^*}b_f = \bm{1}_{\Z(r(e),e)\Z(f,r(f))} = 
	\delta_{e,f}	\bm{1}_{Z(r(e))} = \delta_{e,f} a_{r(e)} \tag{CK1}, \displaybreak[0] \\
	\bm{1}_{Z(u)} = \bm{1}_{\bigsqcup_{e \in uE^1}Z(e)} = \sum_{e \in uE^1}\bm{1}_{Z(e)} = \sum_{e \in uE^1}b_e b_{e^*}. \tag{CK2}
	\end{gather*}
	By the universal property of Leavitt path algebras, there is a unique homomorphism of $R$-algebras $\pi:~ L_R(E) \to A_R(\G_E)$ \label{pi}  such that 
	\begin{align*}
	&\pi(v) = a_v, &&\pi(e) = b_e, &&\pi(e^*) = b_{e^*},
	\end{align*}
	for all $v \in E^0$ and $e \in E^1$. Evidently $\pi$ is a graded homomorphism. The Graded Uniqueness Theorem for Leavitt path algebras implies $\pi$ is injective.
	For a path $\mu \in E^\star$, if we define
	$b_\mu = b_{\mu_1}\dots b_{\mu_{|\mu|}}$ and 
	$b_{\mu^*} = b_{\mu_{|\mu|}^*}\dots b_{\mu_1^*}$
	then it turns out that
	$b_\mu = \bm{1}_{\Z(\mu, r(\mu))}$ and
	$b_{\mu^*} = \bm{1}_{\Z(r(\mu), \mu)}$.
	Moreover, if $\nu \in E^\star$ is another path with $r(\mu) = r(\nu)$, then $b_\mu b_\nu^* = \bm{1}_{\Z(\mu, \nu)}$.
	If $F \subseteq_{\rm finite} r(\mu)E^1$, this yields
	\begin{equation} \label{correspondence formula}
	\bm{1}_{\Z(\mu, \nu, F)} = \bm{1}_{\Z(\mu, \nu)} - \sum_{e \in F} \bm{1}_{\Z(\mu e, \nu e)} = b_{\mu}b_{\nu^*} - \sum_{e \in F} b_{\mu e}b_{{e}^*\nu^*} = \pi\Big(\mu \nu^* - \sum_{e \in F} \mu e e^* \nu^*\Big).
	\end{equation}
	Therefore, $\bm{1}_{\Z(\mu, \nu, F)}$ is in the image of $\pi$. Corollary \ref{cos span} implies that $A_R(\G)$ is generated by functions of the form (\ref{correspondence formula}). We conclude that $\pi$ is surjective. Therefore, $\pi$ is an isomorphism.
\end{proof}

In the following, we generalise \cite[Propositions 3.4 \& 4.9]{tomforde2011leavitt} and \cite[Lemmas 1.5 \& 1.6]{goodearl2009leavitt} by removing restrictions on the graph and the base ring.

\begin{corollary} \label{basic facts}
	Let $E$ be a graph and $R$ a unital commutative ring. Then
	\begin{enumerate}[\rm (1)]
		\item \label{bf 1}
		$L_R(E)$ has homogeneous local units, and it has a unit if and only if $E^0$ is finite;
		\item \label{bf 2}
		The set $\{\mu, \mu^* \in L_R(E)\mid \mu \in E^\star\}$ is $R$-linearly independent in $L_R(E)$;
		\item \label{bf 3}
		For every $v \in E^0$ and $r \in R \setminus \{0\}$, $rv \ne 0$.
		\item \label{bf 5}
		If $r \mapsto \overline{r}$ is an involution on $R$, then there exists a unique involution $L_R(E) \to L_R(E)$ such that $r\mu\nu^* \mapsto \overline{r}\nu\mu^*$ for every $r \in R$ and $(\mu, \nu) \in E^\star \times_r E^\star$.
	\end{enumerate}
\end{corollary}

\begin{proof} 
	(\ref{bf 1})  From Proposition \ref{units}, $L_R(E)$ has homogeneous local units, and it has a unit if and only if $\partial E$ is compact. Since $\partial E = \bigsqcup_{v \in E^0}Z(v)$, and each $Z(v)$ is compact and open, it is clear that $\partial E$ is compact if and only if $E^0$ is finite.
	
	(\ref{bf 2}) Since $L_R(E) = \bigoplus_{n \in \ZZ} L_R(E)_n$, it suffices to show that $\{\mu  \mid \mu \in E^\star, |\mu| = n \}$ and $\{\mu^*  \mid \mu \in E^\star, |\mu| = n \}$ are linearly independent in $L_R(E)$, for every $n \in \ZZ$. Equivalently, $\{\bm{1}_{\Z(\mu, r(\mu))} \mid \mu \in E^\star, |\mu| = n \}$ and $\{\bm{1}_{\Z(r(\mu), \mu)} \mid \mu \in E^\star, |\mu| = n \}$ are linearly independent in $A_R(\G_E)$, for every $n \in \ZZ$. This is clearly true, since $\Z(\mu, r(\mu)), \Z(\nu, r(\nu)) \ne \emptyset$ and $\Z(\mu, r(\mu)) \cap \Z(\nu, r(\nu)) = \emptyset$ for every $\mu, \nu \in E^\star$ such that $\mu \ne \nu$ and $|\mu| = |\nu|$.
	
	(\ref{bf 3}) This follows directly from (\ref{bf 2}), or just the fact that $Z(v) \ne \emptyset$ for all $v \in E^0$.

	(\ref{bf 5}) The existence follows from Proposition \ref{involution proposition}. The uniqueness follows from the universal property of $L_R(E)$.
\end{proof}

Item (\ref{bf 3}) in Corollary \ref{basic facts} is entirely disarmed by the Steinberg algebra model. It was noticed in the early years of Leavitt path algebras that a nontrivial proof was needed for Corollary \ref{basic facts}~(\ref{bf 3}). The first proofs were written, separately, by Goodearl \cite{goodearl2009leavitt} and Tomforde \cite{tomforde2011leavitt} and they involved a representation of $L_R(E)$ on a free $R$-module of infinite rank $\aleph \ge \card(E^0 \sqcup E^1)$. Here is another result from the early years of Leavitt path algebras.

\begin{proposition}\cite[Proposition 3.5]{abrams2007finite}
	If $E$ is a graph and $\mathbb{K}$ a field, then $L_\mathbb{K}(E)$ is finite-dimensional if and only if $E$ is acyclic and $E^0 \cup E^1$ is finite. In this case, if $v_1, \dots, v_t$ are the sinks and $n(v_i) = |\{\alpha \in E^\star \mid r(\alpha) = v_i\}|$, then
	\[
	L_\mathbb{K}(E) \cong \bigoplus_{i = 1}^t M_{n(v_i)}(\mathbb{K}).
	\]
\end{proposition}

\begin{proof}
	From Proposition \ref{finite dim st} we have that $L_\mathbb{K}(E)$ is finite-dimensional if and only if $\G_E$ is finite and discrete. If $E$ had a cycle $c$, then the isotropy group based at $ccc \ldots \in \partial E$ would be infinite. If either $E^0$ or $E^1$ were infinite, then $\partial E$ would be infinite, because $\partial E = \bigsqcup_{v \in E^0}Z(v) = E^0_{\rm sing} \sqcup \big(\bigsqcup_{e \in E^1} Z(e)\big)$. Thus, $\G_E$ is finite only if $E$ is acyclic and $E^0 \cup E^1$ is finite. Conversely, if $E$ is acyclic and $E^0 \cup E^1$ is finite, then there are no infinite paths, and only finitely many finite paths, so $\G_E$ is finite and discrete. To prove the final sentence, note that there are $t$ orbits of sizes $n(v_1), \dots, n(v_t)$, all with trivial isotropy groups. The structure of $L_\mathbb{K}(E)$ is now apparent from Proposition \ref{finite dim st}.
\end{proof}

\subsection{Uniqueness theorems for Steinberg algebras} \label{steinberg uniq}

Steinberg algebras also support a Cuntz-Krieger Uniqueness Theorem and a Graded Uniqueness Theorem. These were first investigated in \cite{brown2014simplicity} and later improved in \cite{clark2015uniqueness} and \cite{steinberg2016simplicity}.
One can think of the Cuntz-Krieger Uniqueness Theorems as saying that a certain property of a graph, namely Condition (\hyperref[L]{L}), or a certain property of an ample groupoid, namely \textit{effectiveness}, forces a homomorphism to be injective -- provided it does not annihilate any scalar multiples of a local unit. This is interesting as a first example of how a Leavitt path algebra theorem translates into the more general setting of Steinberg algebras.

Briefly, this is the order of events in this section. First, we prove the Graded Uniqueness Theorem for Steinberg algebras of graded ample groupoids. Any groupoid can be graded by the trivial group, and this simple trick obtains the Cuntz-Krieger Uniqueness Theorem for Steinberg algebras. We then use the Cuntz-Krieger Uniqueness Theorem for Steinberg algebras to prove the Cuntz-Krieger Uniqueness Theorem for Leavitt path algebras.

\begin{definitions} \label{effective def}
	An \'etale groupoid is
	\begin{enumerate}[(1)]
		\item \textbf{effective} if $\Iso(\G)^\circ = \G^{(0)}$, where $^\circ$ \index{$^\circ$} denotes the interior in $\G$;
		\item \textbf{topologically principal} if $\{x \in \G^{(0)} \mid {x\G x} = \{x\} \}$ is dense in $\G^{(0)}$.
	\end{enumerate}
\end{definitions}

Recall that a groupoid is called principal if the isotropy group at every unit is trivial. Being topologically principal amounts to having a dense set of units with trivial isotropy groups. Obviously, principal implies topologically principal.
Effective does not imply topologically principal, with counterexamples in \cite[Examples 6.3 and 6.4]{brown2014simplicity}, and topologically principal does not imply effective, with counterexamples in \cite[\S5.1]{clark2018simplicity}.
For a deeper understanding of effective groupoids, the upcoming lemma is essential. We state and prove the lemma for more general groupoids than just ample groupoids, mainly because there was an error in its original proof and this is an opportunity to correct it.

First, some topological comments are needed. Sets with compact closure are called \textit{precompact}.
A locally compact, Hausdorff \'etale groupoid $\G$ need not have a base of compact open bisections, but it does have a base of precompact open bisections \cite{brown2014simplicity}. Indeed, $\G$ has a base of open bisections. Since it is locally compact and Hausdorff, $\G$ has a base of open bisections, each of which is contained in a (necessarily closed) compact set, and thus has compact closure. 


\begin{lemma} \cite[Lemma 3.1]{brown2014simplicity} \label{effective lem}
	Let $\G$ be a locally compact Hausdorff \'etale groupoid. Then the following are equivalent:
	\begin{enumerate}[\rm (1)]
		\item \label{elem 1} $\Iso(\G)\setminus \G^{(0)}$ has empty interior in $\G$;
		\item \label{elem 2} $\G$ is effective;
		\item \label{elem 3} Every nonempty open bisection $B \subseteq \G \setminus \G^{(0)}$ contains a morphism $g \notin \Iso(\G)$;
		\item \label{elem 4} For every compact set $K \subseteq \G \setminus \G^{(0)}$ and every nonempty open $U \subseteq \G^{(0)}$, there exists an open subset $V \subseteq U$ such that $VKV = \emptyset$.
	\end{enumerate}
\end{lemma}
\begin{proof}
	$(\ref{elem 1}) \Rightarrow (\ref{elem 2})$ Since $\G$ is \'etale and Hausdorff, $\G^{(0)}$ is clopen in $\G$, so $\G^{(0)} \subseteq \Iso(\G)^\circ$. Now assume $(\Iso(\G) \setminus \G^{(0)})^\circ = \emptyset$.  If $S \subseteq \Iso(\G)$ is open, then $S$ is a disjoint union of two open sets: $S \cap \G^{(0)}$ and $S \cap (\G \setminus \G^{(0)})$. But $S \cap (\G \setminus \G^{(0)}) \subseteq \big(\Iso(\G) \setminus \G^{(0)}\big)^\circ = \emptyset$, so $S \subseteq \G^{(0)}$. This shows $\Iso(\G) = \G^{(0)}$, which means $\G$ is effective.
	
	$(\ref{elem 2}) \Rightarrow (\ref{elem 3})$ Suppose $\G$ is effective. If $B \subseteq \G \setminus \G^{(0)}$ is an open bisection, then $B \subseteq \Iso(\G)$ implies $B \subseteq \Iso(\G)^\circ = \G^{(0)}$ and therefore $B = \emptyset$.
	
	$(\ref{elem 3}) \Rightarrow (\ref{elem 1})$ If there are no nonempty open bisections contained in $\Iso(\G) \setminus \G^{(0)}$, then there are no nonempty open subsets of $\Iso(\G) \setminus \G^{(0)}$, and therefore $\Iso(\G) \setminus \G^{(0)}$ has empty interior.
	
	$(\ref{elem 3}) \Rightarrow (\ref{elem 4})$ 
	We begin by proving a claim: if $B \subseteq \G \setminus \G^{(0)}$ is an open bisection and $U \subseteq \G^{(0)}$ is open and nonempty, then there exists a nonempty open subset $V \subseteq U$ such that $VBV = \emptyset$. If $UBU = \emptyset$, then set $U = V$ and we are done. Otherwise, $UBU\subseteq B \subseteq \G \setminus \G^{(0)}$ is a nonempty open bisection. Applying (\ref{elem 3}), there exists some $g \in UBU$ with $\dom(g) \ne \cod(g)$. Naturally, $\dom(g), \cod(g) \in U$. By the Hausdorff property, there exist disjoint open sets $W, W' \subseteq U$ with $\cod(g) \in W$ and $\dom(g) \in W'$. Set $V = W \cap \cod(BW')$.  Then $\cod(g) \in V$, so $V$ is nonempty, and
	\[
	VB = \big(W \cap \cod(BW')\big) B = WB \cap \cod(BW')B = WB \cap BW'.
	\]
	The last equality uses the fact that $B$ is a bisection, so $\cod(BW')B = BW'$. Therefore,
	\[
	VBV = (WB \cap BW')V \subseteq (BW')V \subseteq (BW')W = \emptyset,
	\]
	because $W'W = W' \cap W = \emptyset$. This proves the claim.
	
	
	Now, let $K \subseteq \G \setminus \G^{(0)}$ be a compact set, and let $U \subseteq \G^{(0)}$ be open and nonempty. We set out to construct a nonempty open subset $V \subseteq U$ such that $VKV = \emptyset$. The set $K$, being compact, can be covered by finitely many open bisections: $K \subseteq B_1\cup \dots\cup B_n$. The claim in the previous paragraph proves the existence of a nonempty open set $V_1 \subseteq U$, such that $V_1B_1V_1 = \emptyset$. Similarly, there is a nonempty open $V_2 \subseteq V_1$ such that $V_2B_2V_2 = \emptyset$. Inductively, this produces a chain of  open sets 
	$
	\emptyset \ne V_n \subseteq V_{n-1} \subseteq \dots \subseteq V_1 \subseteq U
	$
	such that $V_i B_i V_i = \emptyset$ for $1 \le i \le n$. Setting $V = V_n$, we have
	\[
	V K V \subseteq V(B_1 \cup \dots \cup B_n)V \subseteq V_1 B_1 V_1 \cup \dots \cup V_n B_n V_n = \emptyset.
	\]
	
	$(\ref{elem 4}) \Rightarrow (\ref{elem 3})$ Suppose (\ref{elem 3}) does not hold, so there is a nonempty open bisection $B_0 \subseteq \G\setminus \G^{(0)}$ with $B_0 \subseteq \Iso(\G)$. By shrinking it if necessary, we can assume $B_0$ is precompact. Let $K_0 = \overline{B_0}$, the closure of $B_0$. As $\Iso(\G)$ is closed in $\G$, we have that $K_0 \subseteq \Iso(\G)$. Let $U_0 = \cod(B_0)$ and take any $\emptyset \ne V \subseteq U_0$. Since $K_0 \subseteq \Iso(\G)$, it follows that $V K_0 = K_0 V \ne \emptyset$, so $VK_0V \ne \emptyset$. Therefore (\ref{elem 4}) does not hold, because there is no $V \subseteq U_0$ such that $VK_0V = \emptyset$.
\end{proof}

\begin{remark}
	The original proof of the ``$(\ref{elem 3}) \Rightarrow (\ref{elem 4})$'' part of \cite[Lemma 3.1]{brown2014simplicity}, does not appear to be correct. There are examples for which the set $V$ defined in the proof is empty.
	Fortunately, this problem is resolved by defining $V$ inductively, as we have done in the proof of Lemma \ref{effective lem}.
\end{remark}

\begin{lemma} \cite[Proposition 3.6 (i)]{renault2008cartan} \label{eff top princ}
	If a Hausdorff \'etale groupoid $\G$ is topologically principal, then it is effective.
\end{lemma}
\begin{proof}
	Suppose $\G$ is topologically principal: the set $D = \{x \in \G^{(0)} \mid {^x\G^x} = \{x\}\}$ is dense in $\G^{(0)}$.
	If $U \subseteq \Iso(\G)\setminus \G^{(0)}$ is an open bisection (i.e., open in $\G$) then $\dom(U)$ is an open subset of $\G^{(0)} \setminus D$, but $D$ is dense in $\G^{(0)}$, so $\dom(U) = \emptyset$, which implies $U = \emptyset$. This proves $\Iso(\G) \setminus \G^{(0)}$ has empty interior, which implies $\G$ is effective (noting that the proof of $(\ref{elem 1}) \Rightarrow (\ref{elem 2})$ in Lemma \ref{effective lem} only requires $\G$ to be Hausdorff and \'etale).
\end{proof}

The following result is an analogue of \cite[Corollary 2.2.13]{LPAbook}, and it is just an alternative way of presenting some content from \cite{clark2015uniqueness} and \cite{steinberg2016simplicity}.

\begin{proposition} \label{st-reduction}
	Let $\G$ be a $\Gamma$-graded Hausdorff ample groupoid such that $\G_\ep$ is effective. Given a nonzero homogeneous element $h \in A_R(\G)_\gamma$, there exists $C \in B^{\rm co}_{\gamma^{-1}}(\G)$, nonempty $V \in \mathcal{B}(\G^{(0)})$, and nonzero $r \in R$ such that $\bm{1}_C * h * \bm{1}_V = r \bm{1}_V$.
\end{proposition}

\begin{proof} \textit{Step 1} \cite[Lemma 3.1]{clark2015uniqueness}: We show that there exists $B \in B_{\gamma^{-1}}^{\rm co}(\G)$ such that the function $f = \bm{1}_{B} *h$ is $\ep$-homogeneous and its support has nonempty intersection with $\G^{(0)}$. Applying Lemma \ref{mut disj}, we can write $h = \sum_{i = 1}^n r_i \bm{1}_{D_i}$, where $r_1, \dots, r_n \in R \setminus \{0\}$ and $D_1, \dots, D_n \in B_*^{\rm co}(\G)$ are mutually disjoint. Since the $D_i$ are disjoint and the $r_i$ are nonzero, we can assume each $D_i \subseteq \G_\gamma$. Let $B = D_1^{-1}$ and define $f = \bm{1}_B*h$. Then
	\[
	f = \bm{1}_{B}*h = \sum_{i = 1}^n r_i \bm{1}_B * \bm{1}_{D_i} = \sum_{i = 1}^n r_i \bm{1}_{BD_i} = r_1 \bm{1}_{BB^{-1}} + \sum_{i = 2}^n r_i \bm{1}_{BD_i} \in A_R(\G)_\ep.
	\]
	Note that $BD_1, \dots, BD_n \in B^{\rm co}_\ep(\G)$ are mutually disjoint. Indeed, if $x \in B$ and $y \in D_i$ are composable, then $xy \in BD_j$ implies $y = x^{-1}xy \in B^{-1}B D_j = \dom(B)D_j \subseteq D_j$. But $y \in D_i \cap D_j$ implies $i = j$ because $D_1, \dots, D_n$ are disjoint. To show that $(\supp f) \cap \G^{(0)} \ne \emptyset$, let $x \in B$. Then $xx^{-1} \in BD_i$ if and only if $i = 1$. Consequently, $f(xx^{-1}) = r_1 \ne 0$, so $xx^{-1} \in (\supp f) \cap \G^{(0)}$.
	
	\textit{Step 2} \cite{clark2015uniqueness, steinberg2016simplicity}: We show that there exists $V \in \mathcal{B}(\G^{(0)})$ such that $\bm{1}_V * f * \bm{1}_V = r_1\bm{1}_{V}$, where $f$ is from Step 1. The set $K = (\supp f) \setminus BB^{-1} = BD_2 \cup \dots \cup BD_n$ is a compact subset of $\G_\ep \setminus \G^{(0)}$. Since $\G_\ep$ is effective, Lemma \ref{effective lem}~(\ref{elem 4}) proves that a nonempty open set $V \subseteq BB^{-1} = \cod(B)$ exists such that $VKV = \emptyset$. By shrinking if necessary, we can assume $V$ is compact. This yields
	\[
	\bm{1}_V * f * \bm{1}_V = r_1 \bm{1}_{V(BB^{-1})V} + \sum_{i = 2}^n{r_i} \bm{1}_{V(BD_i)V} = r_1 \bm{1}_V.
	\]
	For completion: set $C = VB$ and $r = r_1$. Then $C \in B^{\rm co}_{\gamma^{-1}}(\G)$, $V \in \mathcal{B}(\G^{(0)})$ is nonempty, $r \in R$ is nonzero, and
	$
	\bm{1}_C * h * \bm{1}_V = \bm{1}_V * \bm{1}_B * h * \bm{1}_V = \bm{1}_V * f * \bm{1}_V = r \bm{1}_V.
	$
\end{proof}

%

We are now in a position to prove the Graded Uniqueness Theorem for Steinberg algebras.

\begin{theorem}[Graded Uniqueness Theorem for Steinberg algebras]
	\cite[Theorem 3.4]{clark2015uniqueness}
	\label{graded uniqueness for Steinberg}
	
	Let $\G$ be a $\Gamma$-graded Hausdorff ample groupoid such that $\G_\varepsilon$ is effective. If $A$ is a $\Gamma$-graded ring and $\phi: A_R(\G) \to A$ is a graded homomorphism with the property that $\phi(r\bm{1}_V) \ne 0$ for every nonempty $V \in \mathcal{B}(\G^{(0)})$ and every $r \in R \setminus \{0\}$, then $\phi$ is injective.
\end{theorem}

\begin{proof}
	The kernel of $\phi$ is a graded ideal. Let $h \in (\ker \phi)_\gamma$. If $h \ne 0$ then, according to Proposition \ref{st-reduction}, there exists a compact open bisection $C \subseteq \G_{\gamma^{-1}}$ and a nonempty compact open set $V \subseteq \G^{(0)}$ such that $\bm{1}_C* h * \bm{1}_V = r \bm{1}_V$ for some $r \ne 0$. Then $\phi(r\bm{1}_V) = \phi(\bm{1}_C) \phi(h) \phi(\bm{1}_V) = 0$, which contradicts the assumption about $\phi$. Therefore $h = 0$, so $(\ker \phi)_\gamma = 0$. Since this is true for every $\gamma\in \Gamma$, $\ker \phi = \bigoplus _{\gamma \in \Gamma} (\ker \phi)_\gamma = 0$.
\end{proof}


\begin{remark}
	If $\G = \G_E$ is the groupoid of a graph $E$, then
	\[
	{\G}_0 = \bigcup\left\{\Z(\alpha, \beta)\mid (\alpha, \beta )\in E^\star\times_r E^\star, |\alpha|=|\beta|\right\}
	\]
	so $\Iso({\G}_0)  = \Iso(\G_0)^\circ =  \G^{(0)}$, which shows that $\G$ satisfies the hypotheses of Theorem \ref{graded uniqueness for Steinberg}. The Graded Uniqueness Theorem for Steinberg algebras is a generalisation of the Graded Uniqueness Theorem for Leavitt path algebras, notwithstanding the fact that the latter theorem is usually called upon to prove that all Leavitt path algebras are Steinberg algebras.
\end{remark}

Any groupoid can be graded by the trivial group $\{\ep\}$. With this observation, we immediately obtain the Cuntz-Krieger Uniqueness Theorem for Steinberg algebras \cite[Theorem 3.2]{clark2015uniqueness} .

\begin{corollary}[Cuntz-Krieger Uniqueness Theorem for Steinberg algebras]
	\label{CK uniqueness for Steinberg}
	$\phantom{t}$	
	
	Let $\G$ be an effective Hausdorff ample groupoid. If $A$ is a ring and $\phi: A_R(\G) \to A$ is a homomorphism with the property that $\phi(r\bm{1}_V) \ne 0$ for every nonempty $V \in \mathcal{B}(\G^{(0)})$ and every $r \in R \setminus \{0\}$, then $\phi$ is injective.
\end{corollary}

We now show how  Condition (\hyperref[L]{L}) translates to the groupoid setting.

\begin{proposition} \label{effective remark}
	If $E$ is a graph, then $\G_E$ is effective if and only if $\G_E$ is topologically principal, if and only if $E$ satisfies Condition (\hyperref[L]{L}).
\end{proposition}

\begin{proof}
	\cite{steinberg2018prime} Assume that $E$ satisfies Condition (\hyperref[L]{L}), so that every closed path has an exit. Then every basic open set in $\partial E$ contains a path that is not eventually periodic. Such paths have trivial isotropy groups in $\G$, by Proposition \ref{GE isotropy}, so $\G^{(0)}$ has a dense subset with trivial isotropy. This implies $\G$ is topologically principal, hence effective, by Lemma \ref{eff top princ}.
	On the other hand, if $E$ does not satisfy Condition (\hyperref[L]{L}), then there exists a cycle $c$ without an exit, and $\G_E$ is not effective because there is an open set: $
	\Z(cc, |c|,c) = \{(ccc \dots, |c|, ccc \dots)\} \subseteq \Iso(\G) \setminus \G^{(0)}$.
\end{proof}

Having proved the Cuntz-Krieger Uniqueness Theorem for Steinberg algebras, we can prove the Cuntz-Krieger Uniqueness Theorem for Leavitt path algebras (see Theorem \ref{early-CK-uniqueness}), once and for all, in its full generality.

\begin{theorem}[Cuntz-Krieger Uniqueness Theorem for Leavitt path algebras] $\phantom{t}$
	
	Let $E$ be a graph satisfying Condition (\hyperref[L]{L}) and let $R$ be a unital commutative ring.
	If $A$ is a ring and $\psi: L_R(E) \to A$ is a homomorphism with the property that $\psi(rv) \ne 0$ for every $v \in E^0$ and every $r \in R \setminus \{0\}$, then $\psi$ is injective.
\end{theorem}

\begin{proof}
	First of all, suppose $r \in R \setminus \{0\}$, $\mu \in E^\star$, and $F$ is a finite proper subset of $r(\mu)E^1$. Let $x = r\mu\mu^* - r\sum_{e \in F} \mu ee^* \mu^*$. Then $0 \ne x \in L_R(E)_0$, so Lemma \ref{reduction} yields $(\alpha, \beta) \in E^\star \times_r E^\star$, $v \in E^0$, and $s \in R\setminus \{0\}$ such that $\alpha^*x \beta = sv$. This implies that $\psi(\alpha^*)\psi(x)\psi(\beta) = \psi(sv) \ne 0$, so $\psi(x) \ne 0$.
	
	By Proposition \ref{effective remark}, the groupoid $\G_E$ is effective. Let $\phi: A_R(\G_E) \to A$ be the map $\phi = \psi \circ \pi^{-1}$, where $\pi: L_R(E) \to A_R(\G_E)$ is the isomorphism from Theorem \ref{steinberg-leavitt}. Suppose $V \subseteq \partial E$ is compact and open, and $r \in R \setminus \{0\}$. We can find $\mu \in E^\star$ and $F \subseteq_{\rm finite} r(\mu)E^1$ such that $Z(\mu, F)$ is a nonempty open subset of $V$. Then $Z(\mu,F)V = Z(\mu,F)\cap V = Z(\mu, F)$, so $r \bm{1}_{Z(\mu,F)} = \bm{1}_{Z(\mu,F)}*r\bm{1}_V$. Noting that $\pi^{-1}\left(r\bm{1}_{Z(\mu,F)}\right) = r\mu\mu^* - r \sum_{e \in F}\mu e e^* \mu^*$, the first paragraph proves that $0 \ne \psi \circ \pi^{-1}(r \bm{1}_{Z(\mu,F)}) =  \phi\left(r\bm{1}_{Z(\mu, F)}\right) = \phi\left(\bm{1}_{\Z(\mu,F)}\right)\phi(r\bm{1}_V)$; consequently $\phi(r \bm{1}_V) \ne 0$. Applying Corollary \ref{CK uniqueness for Steinberg}, the map $\phi$ is injective. Conclude that $\psi = \phi \circ \pi$ is injective.
\end{proof}

\section*{Acknowledgements}

I thank Juana S\'anchez Ortega, for her valuable advice and guidance throughout my Masters degree. I also thank Tran Giang Nam for finding an important error in an earlier version of Theorem \ref{path space topology}, and suggesting a way to fix it. (Any errors that remain are my own responsibility.) Finally, I thank the two examiners of my Masters thesis, Pere~Ara and Aidan Sims, who wrote very insightful comments that led to an improvement of this work.

I acknowledge the support of the National Research Foundation of South Africa.

\end{document}